\documentclass[]{amsart}
\usepackage{amscd,amssymb,amsmath,verbatim,hyperref,
epsfig,latexsym,subfigure, supertabular}
\usepackage[arrow,matrix,graph,frame,poly,arc,tips]{xy}

\newtheorem{thm}{Theorem}[section]

\newtheorem{prop}[thm]{Proposition}
\newtheorem*{ack}{Acknowledgements}

\newtheorem{theo}{Theorem}

\theoremstyle{definition}

\newtheorem{example}[thm]{Example}
\newtheorem{remark}[thm]{Remark}

\theoremstyle{remark}
\newtheorem*{rem}{Remark}

\numberwithin{equation}{section}

\newcounter{first}
\newenvironment{probs}
{%
 \begin{list}
       {\textbf\thefirst.}
       {\usecounter{first} \setlength{\leftmargin}{5pt}
        \setlength{\topsep}{5pt}
        \setlength{\itemsep}{5pt}
       }
}%
{\end{list}}

\newcommand{\A}{{\mathcal A}}

\newcommand{\RR}{{\mathcal R}}
\newcommand{\DD}{{\mathcal D}}
\newcommand{\BB}{{\mathcal B}}

\renewcommand{\aa}{{\mathfrak a}}

\newcommand{\B}{{\mathfrak B}}

\newcommand{\m}{{\mathfrak m}}

\newcommand{\N}{{\mathbb N}}
\newcommand{\Z}{{\mathbb Z}}

\newcommand{\C}{{\mathbb C}}
\renewcommand{\P}{{\mathbb P}}
\renewcommand{\k}{\Bbbk}

\newcommand{\PS}{{P\!\varSigma}}
\newcommand{\twoheadlongrightarrow}{\relbar\joinrel\twoheadrightarrow}

\DeclareMathOperator{\Hilb}{Hilb}

\DeclareMathOperator{\rank}{rank}

\DeclareMathOperator{\gr}{gr}

\DeclareMathOperator{\im}{im}
\DeclareMathOperator{\coker}{coker}
\DeclareMathOperator{\codim}{codim}
\DeclareMathOperator{\spn}{span}
\DeclareMathOperator{\ann}{ann}
\DeclareMathOperator{\lin}{lin}
\DeclareMathOperator{\id}{id}
\DeclareMathOperator{\In}{in}

\begin{document}

\title[Chen ranks and resonance]%
{Chen ranks and resonance}

\author[Daniel C. ~Cohen]{Daniel C. ~Cohen}
\thanks{Cohen supported by NSF 1105439, NSA H98230-11-1-0142}
\address{Department of Mathematics,
Louisiana State University,
Baton Rouge, LA 70803}
\email{\href{mailto:cohen@math.lsu.edu}{cohen@math.lsu.edu}}
\urladdr{\href{http://www.math.lsu.edu/~cohen/}%
{http://www.math.lsu.edu/\~{}cohen}}

\author[Henry K. Schenck]{Henry K. Schenck}
\thanks{Schenck supported by NSF 1068754, NSA H98230-11-1-0170}
\address{Department of Mathematics,
University of Illinois Urbana-Champaign, Urbana, IL 61801}
\email{\href{mailto:schenck@math.uiuc.edu}{schenck@math.uiuc.edu}}
\urladdr{\href{http://www.math.uiuc.edu/~schenck/}%
{http://www.math.uiuc.edu/\~{}schenck}}
\curraddr{Department of Mathematics \& Statistics, Auburn University, Auburn, AL 36849}
\email{\href{mailto:hks0015@auburn.edu}{hks0015@auburn.edu}}
\urladdr{\href{http://webhome.auburn.edu/~hks0015/}%
{http://webhome.auburn.edu/\~{}hks0015/}}

\subjclass[2010]{Primary
20F14,  
14M12;  
Secondary
20F36, 
52C35. 
}

\keywords{Chen group, resonance variety, hyperplane arrangement}

\begin{abstract}
The Chen groups of a group $G$ are the lower central series quotients of the maximal metabelian quotient of $G$.  Under certain conditions, we relate the ranks of the Chen groups to the first resonance variety of $G$, a jump locus for the cohomology of $G$.  In the case where $G$ is the fundamental group of the complement of a complex hyperplane arrangement, our results positively resolve Suciu's Chen ranks conjecture.  We obtain explicit formulas for the Chen ranks of a number of groups of broad interest, including pure Artin groups 
associated to Coxeter groups, 
and the group of basis-conjugating automorphisms of a finitely generated free group.
\end{abstract}
\maketitle

\section{Introduction}
\label{intro}
Let $G$ be a group, with commutator subgroup $G'=[G,G]$, and second commutator subgroup $G''=[G',G']$.  The Chen groups of $G$ are the lower central series quotients $\gr_k (G/G'')$ of $G/G''$.  
These groups were introduced by K.T.~Chen in \cite{Ch51}, so as to provide accessible approximations of the lower central series quotients of a link group.  
For example, if $G=F_n$ is the free group of rank $n$ (the fundamental group of the $n$-component unlink), the Chen groups are free abelian, and their ranks, $\theta_k(G)=\rank \gr_k (G/G'')$, are given by
\begin{equation}
\label{eq:chenfn}
\theta_k(F_n)=(k-1)\cdot\binom{k+n-2}{k} , \quad \text{for $k\ge 2$}.
\end{equation}

While apparently weaker invariants than the lower central series quotients of $G$ itself, the Chen groups sometimes yield more subtle information.  For instance, if $G=P_n$ is the Artin pure braid group, the ranks of the Chen groups distinguish $G$ from a direct product of free groups, while the ranks of the lower central series quotients fail to do so, see \cite{CS1}. In this paper, we study the Chen ranks $\theta_k(G)$ for a class of groups which includes 
all arrangement groups (fundamental groups of complements of complex hyperplane arrangements), and potentially fundamental groups of more general smooth quasi-projective varieties.  We relate these Chen ranks to the first resonance variety of the cohomology ring of $G$.  For arrangement groups, our results positively resolve Suciu's  Chen ranks conjecture, stated in \cite{Su}.

Let $A=\bigoplus_{k=0}^\ell A^k$ be a finite-dimensional, graded, graded-commutative, connected 
algebra over an algebraically closed field $\k$ of 
characteristic $0$.   
For each $a\in A^1$, we have $a^2=0$, so right-multiplication by $a$ defines a cochain complex
\begin{equation} \label{eq:aomoto}
(A,a)\colon \quad
\xymatrix{
0 \ar[r] &A^0 \ar[r]^{a} & A^1
\ar[r]^{a}  & A^2 \ar[r]^{a}& \cdots \ar[r]^{a}
& A^{k}\ar[r]^{a} & \cdots}.
\end{equation}
In the context of arrangements, with $A$ the cohomology ring of the complement, the complex $(A,a)$ was introduced by Aomoto~\cite{Ao}, 
and subsequently
used by Esnault-Schechtman-Viehweg \cite{ESV} and Schechtman-Terao-Varchenko \cite{STV}
in the study of local system cohomology. In this context, if $a \in A^1$ is generic, the cohomology of $(A,a)$ vanishes, except possibly in the top dimension, see Yuzvinsky \cite{Yuz}.

In general, the resonance varieties of $A$, or of $G$ in the case where $A=H^*(G;\k)$, are the cohomology jump loci of the complex $(A,a)$,
\[
\RR^k_d(A) = \{a \in A^1 \mid \dim H^k(A,a) \ge d\},
\]
homogeneous algebraic subvarieties of $A^1$.  
These varieties, introduced by Falk \cite{Fa} in the context of arrangements, are isomorphism-type invariants of the algebra $A$.  
They have 
been the subject of considerable recent interest  
in a variety of areas, 
see, for instance, Dimca-Papadima-Suciu \cite{DPS}, Yuzvinsky \cite{Yuz1}, and references therein. 
We will focus on the first resonance variety of the group $G$, 
\[
\RR^1(G)=\RR^1(H^*(G;\k))=\{a\in  A^1 \mid H^1(A,a)\neq 0\}. 
\]

Assume that $G$ is finitely presented.  
The group $G$ is 
$1$-formal (over the field $\k$) if and only if the Malcev Lie algebra of $G$ is is isomorphic, as a filtered Lie algebra, to the completion with respect to degree of a quadratic Lie algebra (see \cite{PS3} for details), and is said to be a commutator-relators group if it admits a presentation $G=F/R$, where $F$ is a finitely generated free group and $R$ is the normal closure of a finite subset of $[F,F]$.  For any finitely generated $1$-formal group, Dimca-Papadima-Suciu \cite{DPS} show that all irreducible components of the resonance variety $\RR^1_d(H^*(G;\k))$ are linear subspaces of $H^1(G;\k)$.  For a finitely presented,  commutator-relators group, the resonance variety $\mathcal R^1(G)$ may be realized as the variety defined by the annihilator of the linearized Alexander invariant $\B$ of $G$, a module over the polynomial ring $S=\text{Sym}(H_1(G;\k))$, $\mathcal R^1(G)=V(\text{ann}(\B))$, see Section \ref{sec:setup} below. We can thus view $\mathcal R^1(G)$ as a scheme.

Let $A=H^*(G;\k)$, and let $\mu\colon A^1 \wedge A^1 \to A^2$ be the cup product map, $\mu(a\wedge b)=a\cup b$.  
A non-zero subspace $U \subseteq A^1$ is said to be $p$-isotropic with respect to the cup product map if the restriction of $\mu$ to $U\wedge U$ has rank $p$.  For instance, $A^1$ is $0$-isotropic if $G$ is a free group, while $A^1$ is $1$-isotropic if $G$ is the fundamental group of a closed, orientable surface.   
Call subspaces $U$ and $V$ of $A^1$ projectively disjoint if they meet only at the origin, $U\cap V=\{0\}$. 
In the formulas below, we use the convention $\binom{n}{m}=0$ if $n<m$. 
Our main result is as follows.

\begin{theo} \label{thm:main}
Let $G$ be a finitely presented, $1$-formal, commutator-relators group.  Assume that the components of $\mathcal R^1(G)$ are (i) $0$-isotropic, (ii) projectively disjoint, and (iii) reduced (viewing $\mathcal R^1(G)$ as a scheme).  Then, for $k\gg 0$, 
\[
\theta_k(G)= (k-1) \sum_{m\ge 2} h_m  \binom{m+k-2}{k},
\]
where $h_m$ denotes the number of $m$-dimensional components of $\RR^1(G)$.
\end{theo}

\begin{example} \label{ex:bad}
Examples illustrating the necessity of the hypotheses in the theorem include the following.  
Let $[u,v]=uvu^{-1}v^{-1}$ denote the commutator of $u,v\in G$.
\begin{probs}
\item[(a)] \label{item:not1formal}
The Heisenberg group $G=\langle g_1,g_2 \mid [g_1,[g_1,g_2]], [g_2,[g_1,g_2]]\rangle$ is not $1$-formal.  Here, $\mathcal R^1(G)=H^1(G;\k)$ is $0$-isotropic since the cup product is trivial.  But $\theta_k(G) \neq \theta_k(F_2)$. Since $G$ is nilpotent, 
the Chen groups of $G$ are trivial.  See~\cite{DPS}.
\item[(b)] \label{item:not0isotropic}
The fundamental group $G$ 
of a closed, orientable surface of genus $g\ge 2$ is $1$-formal.  But $\mathcal R^1(G)=H^1(G;\k)$ is not $0$-isotropic, and $\theta_k(G) \neq \theta_k(F_{2g})$.  See~\cite{PS}.
\item[(c)] \label{item:notdisjoint}
Let $G=G_\Gamma$ be the right-angled Artin group corresponding to the graph $\Gamma$ with vertex set $\mathcal V=\{1,2,3,4,5\}$ and edge set $\mathcal E=\{12,13,24,34,45\}$. The resonance variety $\mathcal R^1(G)$ is the union of two $3$-dimensional subspaces in $H^1(G;\k)$ which are not projectively disjoint, and $\theta_k(G) \neq 2 \theta_k(F_3)$. 
See \cite{PS2}.
\item[(d)] \label{item:notreduced}
Let $G=\langle g_1,g_2,g_3,g_4 \mid [g_2,g_3], [g_1,g_4], [g_3,g_4], [g_1,g_3][g_4,g_2]\rangle$. 
As a variety, $\mathcal R^1(G)$ is a $2$-dimensional subspace of $H^1(G;\k)$.  But $\mathcal R^1(G)$ is not reduced, and $\theta_k(G) \neq \theta_k(F_2)$.  We do not know if this group is $1$-formal.  See Example \ref{ex:notred2}.
\end{probs}
\end{example}

Groups which satisfy the hypotheses of Theorem \ref{thm:main} include all arrangement groups. More generally, let $X$ be a smooth quasi-projective variety, and assume that $G=\pi_1(X)$ is a commutator-relators group.  If $X$ admits a smooth compactification with trivial first Betti number, 
work of Deligne \cite{Del} and Morgan \cite{Mor} implies that $G$ is $1$-formal.  In this instance, Dimca-Papadima-Suciu \cite{DPS} show that the irreducible components of $\RR^1(G)$ are all projectively disjoint and $0$-isotropic.   
Thus, Theorem \ref{thm:main} applies when the components of $\RR^1(G)$ are reduced.  
Another interesting example is provided by the ``group of loops.''

Let $F_n$ be the free group of rank $n$.  The basis-conjugating automorphism group, or pure symmetric automorphism group, is the group $\PS_n$ of all automorphisms of $F_n$ which send each generator $x_i$ to a conjugate of itself.  
Results of Dahm \cite{Dahm} and 
Goldsmith \cite{Goldsmith} imply that this group may also be realized as the ``group of loops,'' 
the group of motions of a collection of $n$ unknotted, unlinked oriented circles in $3$-space, where each circle returns to its original position. 

\begin{theo} \label{thm:mccoolchen}
For $k\gg 0$, the ranks of the Chen groups of the basis-conjugating automorphism group  $\PS_n$ are
\[
\theta_k(\PS_n)=(k-1)\binom{n}{2}+(k^2-1)\binom{n}{3}.
\]
\end{theo}

Our interest in the relationship between Chen ranks and resonance stems from the theory of hyperplane arrangements.  Let $\A=\{H_1,\dots,H_n\}$ be an arrangement in $\C^\ell$, with complement $M=\C^\ell \smallsetminus \bigcup_{i=1}^n H_i$.  It is well known that the fundamental group $G=\pi_1(M)$ is a commutator-relators group, and is $1$-formal.  Furthermore, in low dimensions, the cohomology of $G$ is isomorphic to that of $M$, $H^{\le 2}(G;\k) \cong H^{\le 2}(M;\k)$, see Matei-Suciu \cite{MS}.  Consequently, the first resonance varieties  
of $G$ and of the Orlik-Solomon algebra $A=H^*(M;\k)$ coincide.  
Falk and Yuzvinsky observed that the irreducible components of $\RR^1(G)$ are $0$-isotropic (see \cite{Fa} and \cite{FY}), and Libgober-Yuzvinsky \cite{LY} showed that these components are projectively disjoint.  
In Section \ref{sec:arr} below, we show that the components of $\RR^1(G)$ are reduced.  
Thus, Theorem \ref{thm:main} yields a combinatorial formula for the Chen ranks of the arrangement group $G$ in terms of the Orlik-Solomon algebra of $\A$.  In \cite{Su}, Suciu conjectured that this formula, the Chen ranks conjecture, holds.

Theorem \ref{thm:main} facilitates the explicit calculation of the Chen ranks of a number of arrangement groups of broad interest.  For example, let $W$ be a finite reflection group, and $\A_W$ the arrangement of reflecting hyperplanes.  The fundamental group $PW$ of the complement of $\A_W$ is the pure braid group associated to $W$. 
If $W=A_n$ is the type A Coxeter group, then $PW=PA_n=P_{n+1}$  is the classical Artin pure braid group on $n+1$ strings, 
whose Chen ranks were determined in \cite{CS1}. 
We complete the picture for the remaining infinite families.

\begin{theo} \label{thm:coxeter}
Let $PA_n$, $PB_n$ and $PD_n$ be the pure braid groups associated to the Coxeter groups $A_n$, $B_n$ and $D_n$.  
For $k\gg 0$, the ranks of the Chen groups of these pure braid groups are
\[
\begin{aligned}
\theta_k(PA_n)&=(k-1) \binom{n+2}{4},\\
\theta_k(PB_n)&=(k-1)\left[16\binom{n}{3}+9\binom{n}{4}\right]+(k^2-1)\binom{n}{2},\\
\theta_k(PD_n)&=(k-1)\left[5\binom{n}{3}+9\binom{n}{4}\right].
\end{aligned}
\]
\end{theo}

\section{Preliminaries} \label{sec:setup}
\subsection{Alexander invariant}\label{subsec:AI}
Let $G$ be a finitely presented group, with abelianization $G/G'$ and $\aa\colon G \to G/G'$ the natural projection.  Let $\Z{G}$ be the integral group ring of $G$, and let $J_G=\ker(\epsilon)$ be the kernel of the augmentation map $\epsilon\colon \Z{G} \to \Z$, given by $\epsilon\bigl(\sum m_g g)=\sum m_g$.  Classically associated to the group $G$ are the $\Z(G/G')$-modules $A_G=\Z(G/G') \otimes_{\Z{G}} J_G$ and $B_G=G'/G''$.  The Alexander module $A_G$ is induced from $J_G$ by the extension of the abelianization map $\aa$ to group rings.  The action of $G/G'$ on the Alexander invariant $B_G$ is given on cosets of $G''$ by $gG' \cdot hG''=ghg^{-1}G''$ for $g\in G$ and $h\in G'$.  These modules, and the augmentation ideal of $\Z(G/G')$, comprise the Crowell exact sequence $0 \to B_G \to A_G \to J_{G/G'}\to 0$.

Let $J=J_{G/G'}$, and consider the $J$-adic filtration $\{J^k B_G\}_{k\ge 0}$ of the Alexander invariant.  Let $\gr(B_G)=\bigoplus_{k\ge 0} J^kB_G/J^{k+1} B_G$ be the associated graded module over the ring $\gr(\Z(G/G'))=\bigoplus_{k\ge 0} J^k/J^{k+1}$.  A basic observation of Massey \cite{Ma} shows that 
\[
\gr_k(G'/G'')=\gr_{k-2}(B_G)
\]
for $k\ge 2$, where the associated graded on the left is taken with respect to the lower central series filtration.  Consequently, the Chen ranks of $G$ are given by
\[
\sum_{k\ge 0} \theta_{k+2}(G) t^k = \Hilb(B_G \otimes\k,t),
\]
for any field $\k$ of characteristic zero.

Now assume that $G=\langle g_1,\dots,g_n \mid r_1,\dots,r_m\rangle$ admits a commutator-relators presentation with $n$ generators, and let $p\colon F_n \to G$ be the natural projection, where $F_n=\langle g_1,\dots,g_n\rangle$.  Set $t_i=\aa \circ p(g_i)$.  The choice of basis $t_1,\dots,t_n$ for $G/G' \cong \Z^n$ identifies the group ring $\Z(G/G')$ with the Laurent polynomial ring $\Lambda=\Z[t_1^{\pm 1},\dots,t_n^{\pm 1}]$.  With this identification, the augmentation ideal is given by $J=(t_1-1,\dots,t_n-1)$.  

Presentations for the Alexander module $A_G$ and Alexander invariant $B_G$ may be obtained using the Fox calculus \cite{Fox}.  For $1\le i \le n$, let $\partial_i = \frac{\partial}{\partial g_i}\colon F_n \to \Z{F_n}$ be the Fox free derivatives.  The Alexander module has presentation
\[
\Lambda^m \xrightarrow{\ D\ } \Lambda^n \longrightarrow A_G \longrightarrow 0,
\]
where $D = \bigl(\aa \circ p \circ \partial_i(r_j)\bigr)$ is the Alexander matrix of (abelianized) Fox derivatives.

Let $(C_*,d_*)$ be the standard Koszul resolution of $\Z$ over $\Lambda$, where $C_0=\Lambda$, $C_1=\Lambda^n$ with basis $e_1,\dots,e_n$, and $C_k=\Lambda^{\binom{n}{k}}$.  The differentials are given by 
$d_k(e_I)=\sum_{r=1}^k (-1)^{r+k}(t_{i_r}-1) e_{I \smallsetminus\{i_r\}}$, where $e_I=e_{i_1}\wedge\dots\wedge e_{i_k}$ if $I=\{i_1,\dots,i_k\}$.  
By the fundamental formula of Fox calculus \cite{Fox}, we have $d_1\circ D=0$.  This yields a chain map
\begin{equation} \label{eq:chainmap1}
 \begin{CD}
    @. @. \Lambda^m                         @>D>> \Lambda^n     @>{d_1}>> \Lambda\\
   @.  @.                 @V\alpha VV                           @V\id VV                                                 @V\id VV                      \\
\cdots  @>{d_4}>> \Lambda^{\binom{n}{3}} @>{d_3}>> \Lambda^{\binom{n}{2}} @>{d_2}>>  \Lambda^n           @>{d_1}>> \Lambda  
 \end{CD}
\end{equation}
Basic homological algebra insures the existence of the map $\alpha$ satisfying $d_2\circ \alpha = D$.
Analysis of the mapping cone of this chain map as in \cite{Ma,CSai} then yields a presentation for the Alexander invariant $B_G=\ker(d_1)/\im(D)$:
\[
\Lambda^{\binom{n}{3}+m} \xrightarrow{\ \Delta=d_3+\alpha\ } \Lambda^{\binom{n}{2}} \longrightarrow B_G \longrightarrow 0.
\]

\subsection{Linearized Alexander invariant} The ring $\Lambda=\Z[t_1^{\pm 1},\dots,t_n^{\pm 1}]$ may be viewed as a subring of the formal power series ring $P=\Z[[x_1,\dots,x_n]]$ via the Magnus embedding, defined by $\psi( t_i)= 1+x_i$.  Note that the augmentation ideal $J=(t_1-1,\dots,t_n-1)$ is sent to the ideal $\m=(x_1,\dots,x_n)$.  Passing to associated graded rings, with respect to the filtrations by powers of $J$ and $\m$ respectively, the homomorphism $\gr(\psi)$ identifies $\gr(\Lambda)$ with the polynomial ring $S=\gr(P)=\Z[x_1,\dots,x_n]$.

For $q\ge 0$, let $\psi^{(q)}\colon \Lambda \to P/\m^{q+1}$ denote the $q$-th truncation of $\psi$.  Since $G$ is a commutator-relators group, all entries of the Alexander matrix $D$ are in the augmentation ideal $J$.  It follows that $\psi^{(0)}(D)$ is the zero matrix.  Consequently, all entries of the linearized Alexander matrix $\psi^{(1)}(D)$ are in $\m/\m^2=\gr_1(P)$, so may be viewed as linear forms in the variables of $S$.

For the $n$-generator, commutator-relators group $G$, the polynomial ring $S=\Z[x_1,\dots,x_n]$ may be identified with the symmetric algebra on $H_1(G)=G/G'=\Z^n$.  Let $E=\bigwedge H_1(G)$ be the exterior algebra (over $\Z$), 
and let $(E\otimes S,\partial_*)$ be the Koszul complex of $S$.  Observe that $\partial_k=\psi^{(1)}(d_k)$, in particular, $\partial_1$ is the matrix of variables of $S$, with image the ideal $\m=(x_1,\dots,x_n)$. Note also that $\partial_1 \circ \psi^{(1)}(D)=0$.  The linearized Alexander invariant of $G$ is the $S$-module $\B_\Z=\ker(\partial_1)/\im( \psi^{(1)}(D))$.
\begin{thm}[Papadima-Suciu \cite{PS} Cor. 9.7] Let $G$ be a $1$-formal, commutator-relators group with associated linearized Alexander invariant $\B_\Z$, and let $\k$ be a field of characteristic zero.  Then
\[
\sum_{k\ge 0} \theta_{k+2}(G) t^k = \Hilb(\B_\Z \otimes\k,t).
\]
\end{thm}

A presentation for the linearized Alexander invariant may be obtained by a procedure analogous to that used to obtain one for the Alexander invariant itself.  Linearizing the equality $d_2\circ \alpha = D$, we obtain $\partial_2 \circ \alpha_2 = \psi^{(1)}(D)$, where $\alpha_2=\psi^{(0)}(\alpha)$, the entries of which are in $S/\m$.  Thus, $\alpha_2$ is induced by a map $\Z^m \to \Z^{\binom{n}{2}}$, which we denote by the same symbol. These considerations yield a chain map
\begin{equation} \label{eq:chainmap2}
 \begin{CD}
    @. @. S^m                         @>D^{(1)}>> S^n     @>{\partial_1}>> S\\
   @.  @.                 @V\alpha_2 VV                           @V\id VV                                                 @V\id VV                      \\
\cdots  @>{\partial_4}>> S^{\binom{n}{3}} @>{\partial_3}>> S^{\binom{n}{2}} @>{\partial_2}>>  S^n           @>{\partial_1}>> S  
 \end{CD}
\end{equation}
where $D^{(1)}=\psi^{(1)}(D)$.  
Analysis of the mapping cone of this chain map then yields a presentation for the linearized Alexander invariant:
\begin{equation} \label{eq:linalexpres}
S^{\binom{n}{3}+m} \xrightarrow{\ \Delta^{\lin}=\partial_3+\alpha_2\ } S^{\binom{n}{2}} \longrightarrow \B_\Z \longrightarrow 0.
\end{equation}

\subsection{Resonance}
In the case where $G$ is an arrangement group, it is known \cite{CScv} that the annihilator of the linearized Alexander invariant defines the first resonance variety.  We show that this holds for an arbitrary commutator-relators group.  Let $\k$ be a field of characteristic zero, and consider $\B:=\B_\Z\otimes \k$, the linearized Alexander invariant with coefficients in $\k$.  Abusing notation, let $S=\k[x_1,\dots,x_n]$.

\begin{thm} \label{thm:ann}
Let $G$ be an $n$ generator, commutator-relators group.  The resonance variety $\RR^1(G)$ is the variety defined by the annihilator of the linearized Alexander invariant $\B$, $\RR^1(G) = V(\ann \B)$.
\end{thm}
\begin{proof}
Let $X$ be the $2$-dimensional CW-complex corresponding to a commutator-relators presentation of $G$ with $n$ generators, and let $\RR^1(X)\subset \k^{n}$ be the first resonance variety of the cohomology ring $H^*(X;\k)$.  We first show that $\RR^1(G)=\RR^1(X)$.

A $K(G,1)$ space may be obtained from $X$ by attaching cells of dimension at  least $3$.  The resulting inclusion map $X \hookrightarrow K(G,1)$ induces an isomorphism between $H^1(G;\k)$ and $H^1(X;\k)$. Identify these cohomology groups.  Dualizing the Hopf exact sequences reveals that $H^2(G;\k) \hookrightarrow H^2(X;\k)$. Consequently, if $a,b\in H^1(G;\k)=H^1(X;\k)$, then $a\cup b=0$ in $H^2(G\;k)$ if and only if 
$a\cup b=0$ in $H^2(X;\k)$.  It follows that $\RR^1(G)=\RR^1(X)$.

Consider the chain map \eqref{eq:chainmap2}, now with $S=\k[x_1,\dots,x_n]$.  For $f$ a map of free $S$-modules and $a \in \k^n$, denote the evaluation of $f$ at $a$ by $f(a)$. The resonance variety $\RR^1(G)=\RR^1(X)$ may be realized as the variety in $H^1(G;\k)=\k^{n}$ defined by the vanishing of the $(n-1)\times (n-1)$ minors of the linearized Alexander matrix $D^{(1)}$, see \cite{MS,Su1}.  In other words, 
\[
\RR^1(G)=\{ a \in \k^{n} \mid \rank D^{(1)}(a) <n-1\}.
\]
An exercise with the mapping cone of \eqref{eq:chainmap2} reveals that $\rank D^{(1)}(a) <n-1$ if and only if $\rank \Delta^{\lin}(a)<\binom{n}{2}$.  This implies that $\RR^1(G) = V(\ann \B)$.
\end{proof}

Over the field $\k$, the presentation \eqref{eq:linalexpres} of the linearized Alexander invariant may be simplified.  
Recall that $X$ is the presentation $2$-complex corresponding to an $n$-generator, commutator-relators presentation of $G$. Let $H=H_1(X;\Z)=H_1(G;\Z)=\Z^n$. Specializing \eqref{eq:chainmap2} at $x_i=0$ (resp., \eqref{eq:chainmap1} at $t_i=1$) yields $\alpha_2\colon H_2(X;\k) \to H_2(H;\k)$, which is dual to the cup product map $\mu\colon H \wedge H\to H^2(X;\k)$, see \cite{MS}.   %
The cohomology ring $E=H^*(H;\k)$ is an exterior algebra. 
Let $I \subset E$ be the ideal generated by $\ker(\mu\colon E^1 \wedge E^1 \to H^2(X;\k))$.  Writing $A=E/I$, for $a \in A^1=E^1$, we obtain a short exact sequence of cochain complexes 
$0 \to (I,a) \to (E,a) \to (A,a) \to 0$.  If $a\neq 0$, $(E,a)$ is acyclic, and $H^1(A,a)\neq 0$ if and only if $\ker(I^2\xrightarrow{\ a\ }I^3) \neq 0$.  Thus,
\[
\RR^1(G) = \{ a \in A^1 \mid \ker(I^2\xrightarrow{\ a\ }I^3) \neq 0\}.
\]

If $e_1,\dots,e_n\in E^1$ generate the exterior algebra $E$, let $\delta^k\colon E^{k-1}\otimes S \to E^k\otimes S$ denote multiplication by $\sum_{i=1}^n x_ie_i$, dual to the Koszul differential $\partial_k\colon E^k\otimes S \to E^{k-1}\otimes S$.  Tensoring the exact sequence $0\to I \to E \to A\to 0$ with $S$, we obtain 
a commutative diagram of cochain complexes, with exact columns
\begin{equation} \label{eq:IEA}
 \begin{CD}
\cdots @<{\delta^4_A}<< A^3\otimes S @<{\delta^3_A}<<A^2\otimes S  @<{\delta^2_A}<< A^1\otimes S     @<{\delta^1_A}<< A^0\otimes S\\
   @.  @AAA                 @A\mu\otimes\id AA                           @A\id AA                                                 @A\id AA                      \\
\cdots  @<{\delta^4}<< E^3\otimes S @<{\delta^3}<< E^2\otimes S @<{\delta^2}<<  E^1\otimes S           @<{\delta^1}<< E^0\otimes S\\
   @.        @A AA                                                 @A AA                      \\
\cdots  @<{\delta^4_I}<< I^3\otimes S @<{\delta^3_I}<< I^2\otimes S 
 \end{CD}
\end{equation}
where $\delta^k_I$ denotes the restriction, and $\delta^k_A$ denotes the induced map on the quotient.
Referring to the diagram \eqref{eq:chainmap2}, we have $A_2 \otimes S\cong S^m/\ker(\alpha_2\colon S^m \to S^{\binom{n}{2}})$.  The map $D^{(1)}=\psi^{(1)}(D)$ is trivial on this kernel, let $\bar{D}^{(1)}$ denote the induced map on the quotient.  Then the linearized Alexander invariant may be realized as $\B=\ker(\partial_1)/\im(\bar{D}^{(1)})$, and the map $\delta^2_A$ is dual to $\bar{D}^{(1)}$.  Dualizing the diagram \eqref{eq:IEA}, we obtain an exact sequence of chain complexes 
\[
0 \to (A\otimes S,\partial^A) \to (E \otimes S,\partial^E) \to (I \otimes S,\partial^I) \to 0, 
\]
where $\partial_\bullet$ is dual to $\delta^\bullet$. Since $(E\otimes S,\partial^E)$ is acyclic, we have $\B=H_1(A\otimes S,\partial^A)=H_2(I\otimes S,\partial^I)$, yielding the following presentation for $\B$:
\begin{equation} \label{eq:smalllinalexpres}
I^3\otimes S \xrightarrow{\ \partial_3^I\ } I^2\otimes S \longrightarrow \B \longrightarrow 0.
\end{equation}

\section{Chen ranks from resonance} \label{sec:proveA&B}
In this section, we prove Theorem \ref{thm:main}.  For a group $G$ satisfying the hypotheses of Theorem \ref{thm:main}, we determine the ranks of the Chen groups of $G$ from the resonance variety $\RR^1(G)$,   

Let $G$ be a finitely presented, $1$-formal, commutator-relators group.  These assumptions on $G$ insure (i) that the ranks of the Chen groups are given by the Hilbert series of the linearized Alexander invariant $\B$ of $G$ (with coefficients in $\k$):
\[
\sum_{k\ge 2} \theta_k(G) t^k = \Hilb(\B,t),
\]
see Papadima-Suciu \cite[Cor.~9.7]{PS}; and (ii) that all irreducible components of $\RR^1(G)$ are linear subspaces of $H^1(G;\k)$, see Dimca-Papadima-Suciu \cite[Thm.~B]{DPS}.
Assume that these components of $\RR^1(G)$ are $0$-isotropic, projectively disjoint, and reduced.  

Let $L$ be an irreducible component of $\RR^1(G)$, and let $\{l_1,\dots,l_m\}$ be a basis for $L$, where the $l_i$ are linearly independent elements of $A^1=E^1=H^1(G;\k)$.  Since $L$ is $0$-isotropic, we have $l_i \wedge l_j=0$ in $A^2$ for $1\le i<j\le m$.  Consequently, the ideal 
\begin{equation} \label{eq:ILdef}
I_L = \langle l_i \wedge l_j \mid 1\le i<j\le m\rangle 
\end{equation}
is a subideal of the ideal $I$ generated by 
$\ker(\mu\colon E^1 \wedge E^1 \to A^2)$, where $A^2=H^2(X;\k)$ and $X$ is the presentation $2$-complex corresponding to an $n$-generator, commutator-relators presentation of $G$ as in the previous section.  Corresponding to $L$, we have a ``local'' linearized Alexander invariant $\B_L$, presented in analogy with \eqref{eq:smalllinalexpres} by
\begin{equation} \label{eq:BBqpres}
I^3_L\otimes S \xrightarrow{\ \partial_3^L\ } I^2_L\otimes S \longrightarrow \B_L \longrightarrow 0,
\end{equation}
where $\partial_k^L$ is dual to the map $\delta^k_{I_L}\colon I_L^{k-1} \otimes S \to I_L^k \otimes S$ given by the restriction to $I_L$ of multiplication by $\sum_{i=1}^n x_i e_i$.

The inclusion $I_L \subseteq I$ yields a short exact sequence of cochain complexes of free $S$-modules
\begin{equation} \label{eq:IIq}
 \begin{CD}
\dots @<<<I^4/I_L^4 \otimes S @<{\psi^4}<< I^3/I_L^3 \otimes S @<{\psi^3}<< I^2/I_L^2 \otimes S \\
@.  @AAA      @AAA             @AAA                      \\
\dots @<<<I^4\otimes S  @<{\delta^4_I}<<I^3\otimes S  @<{\delta^3_I}<< I^2\otimes S \\
@.  @AAA       @AAA            @AAA                      \\
\dots @<<<I_L^4\otimes S  @<{\delta^4_{I_L}}<<I_L^3\otimes S  @<{\delta^3_{I_L}}<< I_L^2\otimes S,
 \end{CD}
\end{equation}
where $\psi$ is the induced map. Dualizing and passing to homology yields a surjection
\begin{equation}\label{eq:SurjBB} 
\pi_L \colon \B = \coker(\partial^I_3) \twoheadlongrightarrow \coker(\partial_3^L)= \B_L.
\end{equation}

Let $L_1,\dots,L_k$ be the irreducible components of the resonance variety $\RR^1(G)$, 
and define $\pi \colon \B \to \bigoplus_{i=1}^k \B_{L_i}$ by $\pi =  \bigoplus_{i=1}^k \pi_{L_i}$.  This yields an exact sequence
\begin{equation}
\label{eq:BBq}
0 \longrightarrow {\mathfrak K} \longrightarrow \B \xrightarrow{\ \pi\ }  \bigoplus_{i=1}^k \B_{L_i} \longrightarrow {\mathfrak C} \longrightarrow 0.
\end{equation}
Let $\mathfrak m=\langle x_1,\dots,x_n\rangle$ be the maximal ideal in the polynomial ring $S$.  We will show that the modules ${\mathfrak K}=\ker \pi$ and $ {\mathfrak C}=\coker \pi$ are supported only at $\mathfrak m$.

For $1\le j \le k$, let $\mathfrak q_j=I(L_j)$ be the prime ideal in $S$ with $V(\mathfrak q_j)=L_j$. 
Since localization is 
an exact functor, the sequence  \eqref{eq:BBq} remains exact after localizing at any ideal ${\mathfrak q}$ such that $V({\mathfrak q}) \subset \RR^1(G)$. The assumption that the components of $\RR^1(G)$ are projectively disjoint implies that $(\B_{L_i})_{{\mathfrak q}_j}=0$
if $i \ne j$.  This implies that $ {\mathfrak C}_{\mathfrak q_j}=0$ for $1\le j \le k$.  For an embedded prime $\mathfrak p$ for $\RR^1(G)$, different from $\mathfrak m$, the same argument shows that $ {\mathfrak C}_{\mathfrak p}=0$.  Thus $ {\mathfrak C}$ can be supported
only at the maximal ideal. In the case  
where $G$ is an arrangement group, this yields the lower bound on the 
Chen ranks found in \cite{SS2}; however the proof there relies on a result of 
Eisenbud-Popescu-Yuzvinsky \cite{EPY} which is special to the case of arrangements. 

To establish Theorem \ref{thm:main}, it remains to show that the kernel ${\mathfrak K}$ in \eqref{eq:BBq} is supported only at the maximal ideal.  Let $L=V(\mathfrak q)$ be an $m$-dimensional irreducible component of $\RR^1(G)$.  Given $L\subset\RR^1(G)$, let $J_L$ be the ideal in the exterior algebra $E$ with generating set $\{g \in I^2 \mid l\wedge g \in I_L\ \forall l \in L\}$.  If $L$ has basis $\{l_1,\dots,l_m\}$, then from the definition \eqref{eq:ILdef} of $I_L$, we have $I_L: \langle l_1,\dots,l_m\rangle = \langle l_1,\dots,l_m\rangle$ in $E$. 
Moreover, since $J_L = (I_L: \langle l_1,\dots,l_m\rangle) \cap I$, 
we have 
\begin{equation} \label{eq:intersect}
J_L = \langle l_1,\dots,l_m\rangle \cap I.
\end{equation}
\begin{prop} \label{prop:crux} The ideals $I_L$ and $J_L$ are equal if and only if 
$(\B)_{\mathfrak q} \simeq  (\B_{L})_{\mathfrak q}$.
\end{prop}
\begin{proof}
Choose a basis for $E^1$ so 
that $L=\spn\{e_1,\dots,e_m\}$ and $I_L$ is consequently generated by $e_i e_j:=e_i \wedge e_j$, $1\le i<j\le m$. 

First, suppose $I_L=J_L$. Since $\langle e_1,\dots,e_m\rangle \cap I= I_L$, 
we may choose (independent) elements $\{f_1,\ldots,f_r\}$ in $E^2$ whose images form a basis for the quotient 
$I^2/I^2_L$, and whose initial terms $\In(f_i) = g_i$ are distinct elements of the ideal generated by
$e_{i}e_j$, $m<i<j\le n$. Note that $\{e_1f_1,\ldots, e_1f_r\}$ are independent in $I^3$, for if
not, then 
\[
\sum\limits_{i=1}^r c_ie_1f_i = e_1 \sum\limits_{i=1}^r c_if_i = 0,
\]
so $\sum_{i=1}^r c_if_i \in \langle e_1\rangle$. Since
the initial terms $\In(f_i)=g_i$ are distinct, $\sum_{i=1}^r c_if_i = 0$ is impossible, 
and $\sum_{i=1}^r c_if_i$ cannot be a nontrivial multiple of $e_1$ 
for the same reason. 

Recall from \eqref{eq:IIq} that $\psi^3\colon  I^2/I^2_L\otimes S \to I^3/I^3_L\otimes S$ 
denotes the map induced by $\delta^3_I\colon I^2\otimes S \to I^3\otimes S$ given by multiplication by $\sum_{i=1}^n x_i e_i$.  Let $M$ be the submatrix of (the matrix of) $\psi^3$ with rows corresponding to $e_1f_1,\ldots, e_1f_r \in I^3$.
Then
\[
M = x_1 \cdot \id + M', \mbox{ with } M' \in \k[x_2,\ldots,x_n]^{r \times r},
\]
where $\id$ is the identity matrix of size $r=\dim I^2/I^2_L$. To see this, note that 
\[
\delta^3_I(f_j) = x_1e_1f_j + \sum\limits_{i=2}^n x_ie_if_j.
\]
While there can be syzygies with $e_if_j = e_1f_k + \cdots$, for such a relation we have $i \ge 2$,
hence only $x_i$ with $i \ge 2$ will appear. Thus, $\det(M) = x_1^r +h$, with $\deg_{x_1}(h) < r$.

In the localization $S_{\mathfrak q}$, since $x_1 \not\in \mathfrak q = \langle x_{m+1},\ldots x_n \rangle$, 
$\det(M)$ is a unit. This implies that the maximal Fitting ideal 
$\mathrm{Fitt}_0((\psi^3)_{\mathfrak q})$ 
contains
a unit. Consequently, after localizing, the cokernel of $\psi^3$ vanishes.  From the long exact sequence arising from the dual of \eqref{eq:IIq},
this cokernel is the kernel of the map $(\B)_{\mathfrak q} \rightarrow (\B_L)_{\mathfrak q}\rightarrow 0$, so the map $\pi_L$ of \eqref{eq:SurjBB} induces an isomorphism 
$(\B)_{\mathfrak q} \simeq (\B_L)_{\mathfrak q}$.

Now suppose the irreducible component $L$ of $\RR^1(G)$ is such that $I_L \ne J_L$. Let $g \in (J^2_L \setminus I^2_L)$, assume as 
above that $L$ has basis $\{ e_1,\dots,e_m \}$, and write $L^*=\spn\{ e_{m+1}, \ldots, e_n\}$. Note that $g$ is indecomposable, 
for if $g = l_1 \wedge l_2$ with $l_1 \in L$, $l_2 \in L^*$, then the plane $\spn\{l_1,l_2\} \subseteq R^1(G)$ intersects $L$ in a line,  contradicting the assumption that the components of $\RR^1(G)$ are projectively disjoint.  

Denote the ideal $I_L+g$ in the exterior algebra $E$ by $K_L$, 
and let $\delta^q_{K_L}\colon K_L^{q-1} \otimes S \to K_L^q \otimes S$ be multiplication by $\omega=\sum_{i=1}^n x_i e_i$. 
Let $\B_{K_L}$ be the $S$-module presented by the dual $\partial_3^{K_L}$ of $\delta^3_{K_L}$,
\begin{equation} \label{eq:BKLpres}
I^3_{K_L} \otimes S \xrightarrow{\ \partial_3^{K_L}\ } I^2_{K_L} \otimes S \longrightarrow \B_{K_L}.
\end{equation}
We will construct a free resolution ${\bf K}_\bullet$ of $\B_{K_L}$.

First, consider the cochain complex ${\bf C}^\bullet=(I^\bullet_{L} \otimes S,\delta^\bullet_L)$ corresponding to the ideal $I_L$ in $E$.  This complex is acyclic.  To see this, let ${\bf C}^\bullet(0)$ denote the Koszul (cochain) complex $\bigwedge^* L \otimes S$, with differential given by multiplication by $\omega'=\sum_{i=1}^m x_i e_i$.  Multiplication by $x_{m+1}$ induces a chain map ${\bf C}^\bullet(0) \to {\bf C}^\bullet(0)$.  Let ${\bf C}^\bullet(1)$ be the corresponding mapping cone.  Since ${\bf C}^\bullet(0)$ is acyclic, so is ${\bf C}^\bullet(1)$.  Continuing in this manner, multiplication by $x_{m+j+1}$ induces a chain map ${\bf C}^\bullet(j) \to {\bf C}^\bullet(j)$, and the resulting mapping cone ${\bf C}^\bullet(j+1)$ is acyclic.  Thus, ${\bf C}^\bullet={\bf C}^\bullet(n-m)$ is acyclic.
Let ${\bf C}_\bullet=(I^\bullet_{L} \otimes S,\partial_\bullet^L)$ be the chain complex dual to ${\bf C}^\bullet$.  Since ${\bf C}^\bullet$ is acyclic and free, the dual complex ${\bf C}_\bullet$ gives a free resolution of the module $\B_L$,
\begin{equation} \label{eq:BLres}
\cdots \longrightarrow I^k\otimes S\xrightarrow{\ \partial_k^L\ } \cdots \longrightarrow I^3\otimes S \xrightarrow{\ \partial_3^L\ } I^2\otimes S \longrightarrow \B_L.
\end{equation}

Let $\widehat{\bf C}^\bullet$ be the Koszul complex $\bigwedge^\bullet L^* \otimes S$, with differential given by multiplication by $-\omega''=-\sum_{i=m+1}^n x_i e_i$.  Multiplication by
$\omega' g=\sum_{i=1}^m x_i e_ig$ induces a chain map $\widehat{\bf C}^\bullet \to {\bf C}^\bullet$.
Let ${\bf K}^\bullet$ be the mapping cone of this chain map. Then we have a short exact sequence of cochain complexes $0\to {\bf C}^\bullet \to {\bf K}^\bullet \to \widehat{\bf C}^\bullet \to 0$.  Since ${\bf C}^\bullet$ and $\widehat{\bf C}^\bullet$ are acyclic and free, ${\bf K}^\bullet$ is acyclic and free. 
It is readily checked that the map $\delta^3_{K_L}\colon  I^2_{K_L} \otimes S \to  I^3_{K_L} \otimes S$, $f\mapsto \omega f=\omega' f+\omega'' f$, coincides with the first differential of ${\bf K}^\bullet$.  It follows that the chain complex ${\bf K}_\bullet$ dual to ${\bf K}^\bullet$ is a free resolution of $\B_{K_L}$.

Recall that ${\mathfrak q}=\langle x_{m+1},\dots,x_n\rangle$. The above considerations yield a short exact sequence of chain complexes 
$0\to \widehat{\bf C}_\bullet \to {\bf K}_\bullet \to {\bf C}_\bullet \to 0$, where $\widehat{\bf C}_\bullet$ is dual to $\widehat{\bf C}^\bullet$, and is a free resolution of $S/{\mathfrak q}$.   We thus have a short exact sequence of $S$-modules
\begin{equation} \label{eq:KLseq}
0 \longrightarrow S/{\mathfrak q} \longrightarrow \B_{K_L} \longrightarrow \B_L \longrightarrow 0.
\end{equation}
Since the localization $(S/{\mathfrak q})_{\mathfrak q}$ is nontrivial, we have 
$(\B_{K_L})_{\mathfrak q} \not\simeq (\B_L)_{\mathfrak q}$.  The surjection $\pi_L\colon \B \to \B_L$ of \eqref{eq:SurjBB} factors through $\B_{K_L}$, so the fact that 
$(\B_{K_L})_{\mathfrak q} \not\simeq (\B_L)_{\mathfrak q}$ implies that 
$(\B)_{\mathfrak q} \not\simeq (\B_L)_{\mathfrak q}$, which completes the proof.
\end{proof}

Proposition \ref{prop:crux} implies that if $I_{L}=J_{L}$ for each irreducible component $L$ of $\RR^1(G)$, then 
the modules $\mathfrak K=\ker\pi$ and $\mathfrak C=\coker\pi$ in the exact sequence \eqref{eq:BBq} are of finite length, and hence the Hilbert polynomials of the modules $\B$ and $\bigoplus_{i=1}^k \B_{L_i}$ are equal.  If $\dim L=m$, a straightforward exercise using the resolution \eqref{eq:BLres} shows that the Hilbert polynomial of $\B_L$ is $(k-1)\binom{m+k-2}{k}$.  Consequently, to complete the proof of Theorem \ref{thm:main}, it suffices to prove the following.

\begin{prop} \label{prop:crux2} The ideals $I_L$ and $J_L$ are equal if and only if the irreducible component $L$ of $\RR^1(G)$ is reduced.
\end{prop}
\begin{proof}
We continue with the notation established in the proof of Proposition \ref{prop:crux}. 
Let $L=\spn\{e_1,\dots,e_m\}$ be an $m$-dimensional irreducible component of $\RR^1(G)$, and $\mathfrak q=I(L)$ the prime ideal in $S$ with $V(\mathfrak q)=L$.  The component $L$ of $\RR^1(G)$ is reduced if in the primary decomposition
\[
\ann(\B) = \bigcap Q_i \mbox{ with } \sqrt{Q_i} = \mathfrak{q}_i,
\]
the $\mathfrak{q}$-primary component is equal to $\mathfrak{q}$. 

Assume that $I_L=J_L$. 
For simplicity, 
bigrade $E$, viewing elements of $L$ as of bidegree $(1,0)$ and elements of
$L^*=\spn\{e_{m+1},\dots,e_n\}$ as of bidegree $(0,1)$.  Order monomial bases for 
$I^2_L$ and $I_L^3$ in a grading where $(1,0) > (0,1)$ (e.g., lex order).  Then, 
the map $\partial_3^L\colon I_L^3\otimes S \to I_L^2\otimes S$ of \eqref{eq:BBqpres} presenting the module $\B_L$ has matrix
\begin{equation*} \label{eq:BBqmatrix}
[\partial_3^L]=\left[\begin{matrix}
d^L_2(x_1,\ldots,x_m) & x_{m+1}\cdot \id & x_{m+2}\cdot \id & \cdots & x_n\cdot \id
\end{matrix}\right] = \left[\begin{matrix} d^L_2 & X\end{matrix}\right],
\end{equation*} 
where $d^L_2(x_1,\ldots,x_m)$ is the second Koszul differential on $x_1,\ldots, x_m$, $\id$ is the identity matrix of size $|I_L^2|=\binom{m}{2}$, and $X=\left[\begin{matrix}
x_{m+1}\cdot \id &  \cdots & x_n\cdot \id \end{matrix}\right]$.  Using this, one can check that 
$\ann(\B_L)=\langle x_{m+1},\ldots,x_n\rangle = \mathfrak{q}$.  
Since $(\B)_{\mathfrak q} \simeq (\B_L)_{\mathfrak q}$ when $I_L=J_L$ by Proposition \ref{prop:crux}, and localization commutes with taking annihilators, it follows that $L$ is reduced in this instance.

For the other direction, we show that $I_L \ne J_L$ implies that $L$ is not reduced. 
Let $g$ be an indecomposable element in $J^2_L \setminus I^2_L$, and let $K_L=I_L+g$.
Changing bases in $L$ and $L^*=\spn\{e_{m+1},\dots,e_n\}$, we may assume that 
$g=e_1 e_{m+1}+\dots+e_k e_{m+k}$ 
for some $k$, $2\le k\le m$.  To see that $L$ is not reduced, it suffices to exhibit an element $\beta$ in the module $\B_{K_L}$ which is not annihilated by 
${\mathfrak q}=\langle x_{m+1},\ldots, x_n\rangle$.

Choose ordered bases for $K_L^2$ and $K_L^3$ 
whose initial segments are the bases of $I_L^2$ and $I_L^3$ above, so that $g$ appears
last in the basis for $K^2_{L}$, and $g \wedge L^*$ is the final segment of the basis
for $K^3_{L}$.  With respect to these ordered bases (in light of the mapping cone construction in the proof of Proposition \ref{prop:crux}),  the map $\partial_3^{K_L}\colon K_L^3\otimes S \to K_L^2\otimes S$ of \eqref{eq:BKLpres} presenting $\B_{K_L}$ has matrix
\begin{equation*} \label{eq:BKLmatrix}
[\partial_3^{K_L}]=\left[\begin{matrix} d_2^L & X & 0\\ 0 & \bf{y} & \bf{x}\end{matrix}\right],
\end{equation*}
where ${\bf x} = [x_{m+1}\ \cdots\ x_n]$.  As $\partial_3^{K_L}$ is dual to the multiplication map
$\delta^3_{K_L}\colon K_L^2\otimes S \to K_L^3\otimes S$,  the last row of 
$[\partial_3^{K_L}]$ corresponds to $\delta^3_{K_L}(g) = (\sum_{i=1}^n x_i e_i)\wedge g$.  
Using this, and $g=\sum_{i=1}^k e_i e_{m+i}$, one can check that the the entries of $\bf{y}$ are in $\k[x_1,\dots,x_m]$. 

Let 
$\beta$ be the class of
 $e_1  e_2\in K_L^2\otimes S$ in  $\B_{K_L}$, and 
assume that $x_{m+1} \beta = 0$, i.e., that $x_{m+1}e_1e_2\in \im(\partial_3^{K_L})$. Then there exists $u \in K_L^3 \otimes S$ so that 
$\partial_3^{K_L}(u)=x_{m+1}e_1e_2$.  From the form of the matrix of $\partial_3^{K_L}$ above, we must have $u=e_1e_2e_{m+1}+v$ for some $v \in K_L^3 \otimes S$.  Since the coefficient of $e_1e_2 e_{m+1}$ in $\delta^3_{K_L}(g)$ is $-x_2$, the transpose of the column of $[\partial_3^{K_L}]$ corresponding to $e_1 e_2 e_{m+1}$ is 
$\left[\begin{matrix} x_{m+1} & 0 & \cdots & 0 & -x_2\end{matrix}\right]$.  In other words, 
$\partial_3^{K_L}(e_1e_2e_{m+1})=x_{m+1}e_1e_2-x_2g$.  Thus, 
$\partial_3^{K_L}(v)=\partial_3^{K_L}(u-e_1e_2e_{m+1})=x_2g$, and the assumption that 
$x_{m+1}e_1 e_2\in \im(\partial_3^{K_L})$ implies that $x_2 g\in \im(\partial_3^{K_L})$ as well. 
This in turn implies that the kernel of the map $\B_{K_L} \twoheadrightarrow \B_L$ is annihilated by $x_2$, a contradiction since this kernel is $S/{\mathfrak q}=S/\langle x_{m+1},\dots,x_n\rangle$, see \eqref{eq:KLseq}.
\end{proof}

\begin{example} \label{ex:notred2}
For the group $G$ in Example \ref{ex:bad}(d), $I=\langle e_1e_2,e_1e_3+e_2e_4 \rangle$, and $\RR^1(G)=L=\spn\{e_1,e_2\}$, so that $I_L=\langle e_1e_2\rangle$.  Since $J_L=I \neq I_L$, the resonance component $L$ is not reduced.  A calculation reveals that the Chen ranks of $G$ are given by $\theta_k(G)=2(k-1)$, which of course differ from $\theta_k(F_2)=k-1$ for $k\ge 2$.
\end{example}

\section{Basis-conjugating automorphism groups}
As a first application of Theorem \ref{thm:main}, we compute the Chen ranks of the basis-conjugating automorphism group $G=\PS_n$, proving Theorem \ref{thm:mccoolchen}.

Let $F_n$ be the free group generated by $x_1,\dots,x_n$.  The basis-conjugating automorphism group $\PS_n$ is the group of all automorphisms of $F_n$ which send each generator $x_i$ to a conjugate of itself.  As noted in the introduction, this group may be realized as the group of motions of a collection of $n$ unknotted, unlinked oriented circles in $3$-space, where each circle returns to its original position.   McCool \cite{McC} found the following presentation for the basis-conjugating automorphism group:
\begin{equation}\label{eq:McCool}
\PS_n = \langle \beta_{i,j}, 1\le i \neq j \le n \mid [\beta_{i,j},\beta_{k,l}], [\beta_{i,k},\beta_{j,k}], [\beta_{i,j},(\beta_{i,k}\cdot \beta_{j,k})]
\rangle,
\end{equation}
where $[u,v]=u^{}v^{}u^{-1}v^{-1}$, 
the indices in the relations are distinct, and the generators $\beta_{i,j}$ are the automorphisms of
$F_n$ defined by
\[
\beta_{i,j}(x_k^{})=\begin{cases} x_k^{} &\text{if $k \neq j$,}\\ x_j^{-1} x_i^{} x_j^{}&\text{if $k=i$.}
\end{cases}
\]

The integral cohomology of $\PS_n$ was determined by Jensen-McCammond-Meier \cite{JMM}, resolving a conjecture of Brownstein-Lee \cite{BL}.  We rephrase their result for a field $\k$ of characteristic zero.  Let $E$ be the exterior algebra over $\k$ generated by degree one elements $e_{p,q}$, $1\le p \neq q \le n$, and let $I$ be the two-sided ideal in $E$ generated by
\[
e_{i,j} e_{j,i},\  1\le i< j\le n,\quad   (e_{k,i}-e_{j,i})(e_{k,j}-e_{i,j}),
\  1\le i,j,k \le n,\  i<j,\  k\notin\{i,j\}.
\]
Then the cohomology algebra of the basis-conjugating automorphism group $\PS_n$ is isomorphic to the quotient of $E$ by $I$, $H^*(\PS_n;\k) \cong E/I$.  Using this description of the cohomology, Cohen \cite{Coh} computed the first resonance variety of $\PS_n$:
\begin{equation} \label{eq:mccoolres}
\RR^1(\PS_n) = \bigcup_{1\le i < j \le n}\!\! C_{i,j} \ \cup 
\bigcup_{1\le i < j<k \le n}\!\!\! C_{i,j,k} \subset H^1(\PS_n;\k)=\k^{n(n-1)},
\end{equation}
where $C_{i,j}=\spn\{e_{i,j},e_{j,i}\}$ and $C_{i,j,k}=\spn\{e_{j,i}-e_{k,i},e_{i,j}-e_{k,j},e_{i,k}-e_{j,k}\}$.

\begin{proof}[Proof of Theorem \ref{thm:mccoolchen}]
 From work of Berceanu-Papadima \cite{bp}, it is known that $\PS_n$ is $1$-formal.  And it is clear from \eqref{eq:McCool} that $\PS_n$ is a commutator-relators group.  Checking that the components $C_{i,j}$, $C_{i,j,k}$ from \eqref{eq:mccoolres} are all projectively disjoint, and are all $0$-isotropic, the Chen ranks of $\PS_n$ are given by Theorem \ref{thm:main} provided that all these components are reduced.  
 
 The symmetric group on $n$ letters acts on $\PS_n$ by permuting indices, $\sigma(\beta_{i,j})=\beta_{\sigma(i),\sigma(j)}$, and hence on the cohomology and resonance variety of $\PS_n$.  In light of this action, it suffices to show that the resonance components $C_{1,2}$ and $C_{1,2,3}$ are reduced.  We establish this using Proposition \ref{prop:crux}.
 
Let $L=C_{1,2,3}=\spn\{e_{2,1}-e_{3,1},e_{1,2}-e_{3,2},e_{1,3}-e_{2,3}\}$, and let $g \in I^2$, where $I$ is the ideal in the exterior algebra $E$ defining the cohomology of $\PS_n$.  Write
\[
g=\sum_{i<j} a_{i,j} e_{i,j}e_{j,i}+\sum_{k,i<j}b^k_{i,j}(e_{k,i}-e_{j,i})(e_{k,j}-e_{i,j}),
\]
where $a_{i,j},b_{i,j}^k \in\k$. We will show that if $l\wedge g \in I_L$ for all $l\in L$, then $g \in I_L$.  This implies that $I_L=J_L$, and hence that $L$ is reduced by Proposition \ref{prop:crux}.

Note that $s^1_{2,3}=b^1_{2,3}(e_{2,1}-e_{3,1})(e_{1,2}-e_{3,2})(e_{1,3}-e_{2,3})$ is in $I_L$.  
So $(e_{2,1}-e_{3,1})g \in I_L$ implies that $h=(e_{2,1}-e_{3,1})g - s^1_{2,3}$ is also in $I_L$.  Write
\begin{equation} \label{eq:C123}
\begin{aligned}
h&=c_{2,3}^1(e_{1,2}-e_{3,2})(e_{1,3}-e_{2,3})+c^2_{1,3}(e_{2,1}-e_{3,1})(e_{2,3}-e_{1,3})\\
&\qquad+c^3_{1,2}(e_{3,1}-e_{2,1})(e_{3,2}-e_{1,2}),
\end{aligned}
\end{equation}
where $c^1_{2,3}=\sum x_{i,j} e_{i,j}$, $c^2_{1,3}=\sum y_{i,j} e_{i,j}$, and $c^3_{1,2}=\sum z_{i,j} e_{i,j}$ with $x_{i,j},y_{i,j},z_{i,j}\in\k$. Comparing coefficients of $e_{2,1}e_{i,j}e_{j,i}$ in the left- and right-hand sides of \eqref{eq:C123} reveals that $a_{i,j}=0$ for $\{i,j\} \not\subset\{1,2,3\}$.  Similarly, if $k\notin\{1,2,3\}$ and $i<j$, 
by considering the coefficients of $e_{2,1}e_{k,i}e_{k,j}$, $e_{2,1}e_{k,i}e_{j,k}$, and $e_{2,1}e_{k,j}e_{i,k}$ in \eqref{eq:C123}, we see that $b^k_{i,j}=0$, $b^j_{i,k}=0$ (if $i<k$) or $b^j_{k,i}=0$ (if $k<i$), and $b^i_{j,k}=0$ or $b^i_{k,j}=0$.  Thus, we have $b_{i,j}^k=0$ if $\{i,j,k\}\neq\{1,2,3\}$.

These considerations imply that $h=\sum_{1\le i<j\le 3}a_{i,j}(e_{2,1}-e_{3,1})e_{i,j}e_{j,i}$.  Using \eqref{eq:C123} again, comparing coefficients of $e_{i,j}e_{1,2}e_{1,3}$, $e_{i,j}e_{2,1}e_{2,3}$ and $e_{i,j}e_{3,1}e_{3,2}$ reveals that $x_{i,j}=0$, $y_{i,j}=0$, and $z_{i,j}=0$ for $\{i,j\}\not\subset\{1,2,3\}$.  So, for instance, $c^1_{2,3}=\sum_{1\le i\neq j\le 3}x_{i,j}e_{i,j}$, and similarly for $c^2_{1,3}$ and $c^3_{1,2}$.  Then, comparing coefficients in \eqref{eq:C123} as indicated below yields the following:
\[
\begin{matrix}
e_{1,2}e_{2,3}e_{3,2}\colon x_{1,2}=0&e_{1,3}e_{2,3}e_{3,2}\colon x_{1,3}=0&e_{2,1}e_{2,3}e_{3,2}\colon 
x_{2,1}+y_{3,2}+z_{2,3}=a_{2,3}\\
e_{1,2}e_{1,3}e_{2,3}\colon x_{2,3}=0&e_{1,2}e_{1,3}e_{3,2}\colon x_{3,2}=0&e_{3,1}e_{2,3}e_{3,2}\colon
x_{3,1}+y_{3,2}+z_{2,3}=a_{2,3}
\end{matrix}
\]
It follows that $x_{2,1}+x_{3,1}=0$, and $c^1_{2,3}=x(e_{2,1}-e_{3,1})$, where $x=x_{2,1}$.  Similar considerations yield $c^2_{1,3}=y(e_{1,2}-e_{3,2})$ and $c^3_{1,2}=z(e_{1,3}-e_{2,3})$.

Summarizing, we have
\[
h=\sum_{1\le i<j\le 3}a_{i,j}(e_{2,1}-e_{3,1})e_{i,j}e_{j,i}=\lambda(e_{2,1}-e_{3,1})(e_{1,2}-e_{3,2})(e_{1,3}-e_{2,3}),
\]
where $\lambda=x+y+z$.  Comparing coefficients of, for instance, $e_{1,3}e_{1,2}e_{2,1}$ here yields $\lambda=0$, which implies that $a_{i,j}=0$ for all $i,j$.  Hence, we have
\[
g=b^1_{2,3}(e_{1,2}-e_{3,2})(e_{1,3}-e_{2,3})+b^2_{1,3}(e_{2,1}-e_{3,1})(e_{2,3}-e_{1,3})+b^3_{1,2}(e_{3,1}-e_{2,1})(e_{3,2}-e_{1,2})
\]
and $g\in I_L$.  Thus, $I_L=J_L$, and the component $L=C_{1,2,3}$ of $\RR^1(\PS_n)$ is reduced.

A similar (easier) argument, which we leave to the reader, shows that the component $C_{1,2}$ of $\RR^1(\PS_n)$ is also reduced.  Thus, Theorem \ref{thm:main} may be used to compute the Chen ranks of $\PS_n$.   Noting that $\RR^1(\PS_n)$ has $\binom{n}{2}$ two-dimensional components and $\binom{n}{3}$ three-dimensional components completes the proof of Theorem \ref{thm:mccoolchen}.
\end{proof}

\begin{remark} \label{rem:upper}
The upper triangular McCool groups illustrate the necessity of the hypotheses of Theorem \ref{thm:main}. For each $n\ge 2$, the upper triangular McCool group is the subgroup $\PS_n^+$ of $\PS_n$ generated by the elements $\beta_{i,j}$ with $1\le i<j \le n$, subject to the relevant relations \eqref{eq:McCool}. Thus, $\PS_n^+$ is a commutator-relators group.  Moreover, in \cite{bp}, Berceanu-Papadima remark that $\PS_n^+$ is $1$-formal.

In \cite{CPVW}, Cohen-Pakianathan-Vershinin-Wu determine the integral cohomology of $\PS_n^+$, see also \cite{DCcompo}.  
We rephrase their result for a field $\k$ of characteristic zero.  Let $E^+$ be the exterior algebra over $\k$ generated by elements $e_{p,q}$, $1\le p < q \le n$, of degree one, and let $I^+$ be the two-sided ideal in $E^+$ generated by $e_{j,k}(e_{i,k}-e_{i,j})$, $1\le i<j<k\le n$. Then, $H^*(\PS_n^+;\k) \cong E^+/I^+$.

Using the above description of the cohomology ring, one can check that the resonance variety 
 $\RR^1(\PS_4^+)$ has components
 \[
\begin{aligned}
&\spn\{e_{1,3}-e_{1,2},e_{2,3}\}, \ \spn\{e_{1,4}-e_{1,2},e_{2,4}\}, \text{ and }\\
&\spn\{e_{1,3}-e_{1,4},e_{2,3}-e_{2,4},e_{3,4}\}.
\end{aligned}
\]
These components are projectively disjoint. 
However, the $3$-dimensional component $L=\spn\{e_{1,3}-e_{1,4},e_{2,3}-e_{2,4},e_{3,4}\}$ is neither $0$-isotropic nor reduced. For the former, note that $(e_{1,3}-e_{1,4})(e_{2,3}-e_{2,4})$ is nonzero in $H^*(\PS_4^+;\k)$. For the latter, $L$ has an embedded component $\spn\{e_{3,4}\}$. Accordingly (see Proposition \ref{prop:crux2}), the ideals $I_L$ and $J_L$ are not equal. For instance, $e_{2,4}(e_{1,4}-e_{1,2})-e_{2,3}(e_{1,3}-e_{1,2})$ is in $J_L$, but not in $I_L$.

In light of the above observations, it is not surprising that the Chen ranks formula of Theorem \ref{thm:main} does not hold for the upper triangular McCool groups. For example, a computation reveals that, for $k\gg 0$, the Chen ranks of $\PS_4^+$ are given by $\theta_k(\PS_4^+)=1+\frac{5}{2}\theta_k(F_2)+\frac{1}{2}\theta_k(F_3)$, which differs from the value $2\theta_k(F_2)+\theta_k(F_3)$ naively predicted from the resonance variety $\RR^1(\PS_4^+)$.
\end{remark}

\section{Hyperplane arrangements} \label{sec:arr}
Let $\A=\{H_1,\dots,H_n\}$ be a hyperplane arrangement in $\C^\ell$, with complement $M(\A)=\C^\ell\smallsetminus \bigcup_{i=1}^n H_i$.  We assume that $\A$ is a central arrangement, i.e., that each hyperplane of $\A$ passes through the origin.  Let $L(\A)=\{\bigcap_{H\in\BB}H \mid \BB\subseteq\A\}$ be the intersection lattice of $\A$, with rank function given by codimension.  We refer to elements of $L(\A)$ as flats. Recall that $\k$ is a field of characteristic zero.
A well known theorem of Orlik-Solomon \cite{OS} yields a presentation for the cohomology ring $A=H^*(M(\A);\k)$, the Orlik-Solomon algebra, in terms of the lattice $L(\A)$.  Let $G=\pi_1(M(\A))$ be the fundamental group of the complement.  As noted in the introduction, the arrangement group $G$ is a $1$-formal, commutator-relators group, and the resonance varieties $\RR^1(G)$ and $\RR^1(A)$ coincide.  In this context, we denote this variety by $\RR^1(\A)$, the first resonance variety of the arrangement $\A$.

Falk \cite{Fa} initiated the study of resonance varieties in the context of arrangements. 
Among
his main innovations was the concept of a neighborly
partition.  A partition $\Pi$ of $\A$ is neighborly if, for
any rank two flat $Y\in L_2(\A)$ and any block $\pi$ of $\Pi$,
\[
|Y|-1 \le |Y \cap \pi| \Longrightarrow Y\subseteq \pi, 
\]
Partitions with a single block will be called trivial, others nontrivial. 
Flats contained in a single block of $\Pi$ will be referred to as monochrome, others polychrome. 
Flats of multiplicity two are necessarily monochrome. 

Falk showed that all components of $\RR^1(\A)$ arise from
nontrivial neighborly partitions of subarrangements of $\A$, and conjectured that $\RR^1(\A)$ was a 
subspace arrangement. This was proved simultaneously by Cohen-Suciu in \cite{CScv} and 
Libgober-Yuzvinsky in \cite{LY}, and the latter also showed that the irreducible components of $\RR^1(\A)$ are projectively disjoint.  As noted by Falk-Yuzvinsky \cite{FY} (see also \cite{Fa}), these components are $0$-isotropic.  Arrangements which admit nontrivial neighborly partitions, and corresponding resonance components, include all  central arrangements in $\C^2$, the rank $3$ braid arrangement, the Pappus, Hessian, and type B Coxeter arrangements in $\C^3$, etc., 
see \cite{CScv,Fa, PY,Su} among others.

Let $\Pi$ be a neighborly partition of a subarrangement $\A'$ of $\A$.  Following \cite{Fa, Fa01, LY}, we explicitly describe the corresponding component $L_\Pi$ of $\RR^1(\A)$.  Let $E$ be the exterior algebra over $\k$, with generators $e_1,\dots,e_n$ corresponding to the hyperplanes of $\A$.  The Orlik-Solomon algebra is given by $H^*(M(\A);\k)=A=E/I$, where $I$ is the Orlik-Solomon ideal of $\A$.  This ideal is generated by boundaries of circuits, $\partial e_{i_1}\cdots e_{i_k}$, where $\{H_{i_1},\dots,H_{i_k}\}$ is a minimally dependent set of hyperplanes in $\A$ and $\partial\colon E \to E$ is defined by $\partial 1=0$, $\partial e_i=1$, and $\partial(uv)=(\partial u) v + (-1)^{|u|}u(\partial v)$, $|u|$ denoting the degree of $u$, see \cite[Ch. 3]{OT}. For $u=\sum_{i=1}^n u_i e_i \in E^1$, write $\partial_i u = u_i$, and $\partial_X u = \sum_{X \subset H_i} u_i$ for a rank $2$ flat $X$.  Let $\mathrm{poly}(\Pi)$ denote the set of rank two flats $X$ which are polychrome with respect to $\Pi$.    Then the resonance component corresponding to the neighborly partition $\Pi$ of $\A' \subset \A$ is given by
\begin{equation} \label{eq:res comp}
L_\Pi=\{u \in E^1 \mid \partial u = 0,\ \partial_X u = 0\ \forall X \in \mathrm{poly}(\Pi),\ \partial_i u=0\ \forall H_i \notin \A'\}.
\end{equation}

Since the irreducible components of $\RR^1(\A)$ are $0$-isotropic and projectively disjoint, to bring Theorem \ref{thm:main} to bear in the context of arrangement groups, it suffices to 
show that these components are also reduced. 
To establish this, in addition to Falk's description of components of $\RR^1(\A)$ in terms of neighborly partitions described above, we will also use the more recent Falk-Yuzvinsky \cite{FY} description of resonance components in terms of multinets, which utilizes the Libgober-Yuzvinsky \cite{LY} treatment in terms of the Vinberg classification of generalized Cartan matrices.

A multiarrangement is an arrangement $\A$ together with a multiplicity function $\nu\colon \A \to \N$ which assigns a positive integer $\nu(H)$ to each hyperplane in $\A$. A $(k,d)$-multinet on a multiarrangement $(\A,\nu)$ (of lines in $\C\P^2$) is a pair $(\Pi,\mathcal X)$, where $\Pi=(\beta_1|\beta_2|\cdots|\beta_{k})$ is a partition of $\A$ with $k\ge 3$ blocks, and $\mathcal X$ is a set of rank $2$ flats of $\A$, satisfying
\begin{enumerate}
\item[(i)] For each block $\beta_i$ of the partition $\Pi$, the sum $\sum_{H \in \beta_i} \nu(H) = d$ is constant, independent of $i$;
\item[(ii)] For each $H\in\beta_i$ and $H' \in\beta_j$ with $i\neq j$, the flat $H\cap H'$ is in $\mathcal X$;
\item[(iii)] For each flat $X \in \mathcal X$, $\sum_{H \in \beta_i, X\subset H} \nu(H)$ is constant, independent of $i$;
\item[(iv)] For each $i$ and $H,H' \in \beta_i$, there is a sequence $H=H_0,H_1,\dots,H_r=H'$ of hyperplanes of $\A$ such that $H_{j-1} \cap H_j \notin \mathcal X$ for $1\le j\le r$.
\end{enumerate}

As shown by Falk-Yuzvinsky \cite{FY}, an $\ell$-dimensional component of $\RR^1(\A)$ corresponds to a multiplicity function and an $(\ell+1,d)$-multinet on a subarrangement of $\A$, for some $d$. Let $e_1,\dots,e_n$ be the standard generators of the exterior algebra $E$, in correspondence with the hyperplanes of $\A$. 
If $(\Pi,\mathcal X)$ is a multinet on $(\A',\nu)$, where $\A'\subset \A$ and $\Pi=(\beta_1|\beta_2|\cdots|\beta_{\ell+1})$, the corresponding resonance component $L=L_\Pi$ has basis
\begin{equation} \label{eq:multinetbasis}
\{\nu_1 - \nu_{\ell+1}, \nu_2 - \nu_{\ell+1}, \dots, \nu_\ell - \nu_{\ell+1}\}, \ \text{where $\nu_i = \sum_{H_j \in \beta_i} \nu(H_j) e_j = \sum_{H_j\in\beta_i} \nu_{i,j} e_j$,}
\end{equation}
see \cite[Thms. 2.4, 2.5]{FY}.
Note that $\dim L=\ell$, the partition $\Pi$ is neighborly, the elements of $\mathcal X$ are the polychrome flats of $\Pi$, the coefficients $\nu_{i,j}=\nu(H_j)$ are positive integers, and $\partial \nu_i = \partial \nu_j$ for $1\le i,j\le \ell+1$ by (i) above.

\begin{thm} \label{thm:Ared} For any hyperplane arrangement $\A$, the irreducible components of the resonance variety $\RR^1(\A)$ are reduced.
\end{thm}

\begin{rem} The conclusion reached  at the end of the second paragraph of the proof of this result in the published version of the paper [Adv. Math. \textbf{285} (2015), 1--27] is not valid. We thank Claudiu Raicu for bringing this to our attention, and provide a alternate proof of the theorem here.
\end{rem}

\begin{proof}[Proof of Theorem \ref{thm:Ared}]  
 
We first note that it suffices to consider components of $\RR^1(\A)$ which are supported on the entire arrangement. If $L\subset \RR^1(\A)$ is a component supported on a proper subarrangement $\A'$ of $\A$, and $L$ is nonreduced in $H^1(M(\A);\Bbbk)$, then $L \subset  
H^1(M(\A');\Bbbk) \subset H^1(M(\A);\Bbbk)$ is necessarily nonreduced in $H^1(M(\A');\Bbbk)$
as well.

So let $L \subset\RR^1(\A)$ be an irreducible component which is supported on the entire arrangement $\A$.  By Propositions \ref{prop:crux} and \ref{prop:crux2},
it suffices to show that the ideals $I_L=\bigwedge^2L$ and $J_L=\langle g \in I^2 \mid x  g \in I_L \forall x \in L\rangle$ in the exterior algebra $E$ are equal. 

The component $L$ corresponds to multinet structure on $\A$, with associated multiplicity function $\nu\colon \A \to \N$, neighborly partition $\Pi$, and ${\mathcal X} = \mathrm{poly}(\Pi) \subset L_2(\A)$ the polychrome flats of $\Pi$.  
If the  partition $\Pi$ of $\A$
is given by $\Pi=(\beta_1\mid \beta_2\mid\dots\mid \beta_{\ell+1})$, the resonance component $L$ has basis given by 
\eqref{eq:multinetbasis} above and $\dim L=\ell$.

Recall from \eqref{eq:intersect} that 
$J_L = \langle \nu_1-\nu_{\ell+1},\dots,\nu_\ell-\nu_{\ell+1} \rangle \cap I$.
Consequently, if $g \in J_L$ is a generator (not necessarily in $I_L$), we can write 
\[
g=(\nu_1-\nu_{\ell+1})u_1+\dots+(\nu_\ell-\nu_{\ell+1})u_\ell.
\] 
We will show that $g \in I_L$ by showing 
that $u_i \in L = L_\Pi$ for each $i$, $1\le i\le \ell$, 
that is, $u_i$ satisfies the 
conditions of  \eqref{eq:res comp}. 
Note that the last of these conditions is vacuous.

Since $\nu_i-\nu_{\ell+1}\in L$, we have $\partial (\nu_i-\nu_{\ell+1})=0$.  Similarly, since $g$ is in the Orlik-Solomon ideal $I$, we have $\partial g=0$.  Computing
\[
\partial g = -(\nu_1-\nu_{\ell+1}) \partial u_1-\dots-(\nu_\ell-\nu_{\ell+1})\partial u_\ell=0, 
\]
the fact that $\{\nu_1-\nu_{\ell+1},\dots,\nu_\ell-\nu_{\ell+1}\}$ is linearly independent implies that $\partial u_1=\dots=\partial u_\ell=0$.  It remains to show that $\partial_X(u_i)=0$ for $X \in \mathrm{poly}(\Pi)$ a rank two flat of $\A$ which is polychrome with respect to $\Pi$.

Let $X\in L_2(\A)$ be a a polychrome flat, that is, $X \in \mathcal X$.  Then $X$ meets each block of $\Pi$, see \cite[Rem. 3.10]{LY}.  We can assume that $X$ is contained in the hyperplanes $H_1, \dots, H_{\ell+1}$ of $\A$ (and possibly others), and that $H_j \in \beta_j$ for $1\le j \le \ell+1$.  Write the basis elements of $L$ as $\nu_i-\nu_{\ell+1}=y_i+z_i$, where $y_i\in E^1_X=\spn\{e_r \mid X \subset H_r\}$ and $z_i=\nu_i-\nu_{\ell+1}-y_i$ for each $i$, $1\le i\le \ell$. From our assumptions concerning the hyperplanes containing $X$ and \eqref{eq:multinetbasis}, we have
\[
y_i = \nu_{i,i}e_i - \nu_{\ell+1,\ell+1} e_{\ell+1} + \sum_{k>\ell+1,X\subset H_k} (\nu_{i,k}-\nu_{\ell+1,k})e_k.
\]
Since the multiplicities $\nu_{i,j}=\nu(H_j)$ are positive integers, it is readily checked that $\{y_1,\dots,y_\ell\}$ is linearly independent.

To analyze $\partial_X(u_i)$, first note that, since $\nu_i-\nu_{\ell+1} =y_i+z_i\in L$, we have $\partial_X(\nu_i-\nu_{\ell+1})=0$, which implies that $\partial_X y_i=\partial y_i=0$. This, together with $\partial(\nu_i-\nu_{\ell+1})=0$, yields $\partial z_i=0$ for each $i$. Writing
\[
g=\sum_{i=1}^\ell (\nu_i-\nu_{\ell+1}) u_i=\sum_{i=1}^\ell (y_i+z_i)u_i
=\sum_{i=1}^\ell y_iu_i+\sum_{i=1}^\ell z_iu_i = g_1+g_2
\]
we have $g_1=\sum_{i=1}^\ell y_iu_i \in I$ since $g \in J_L$ is in $I$ and $\partial z_i=\partial u_i=0$ implies that $g_2=\sum_{i=1}^\ell z_iu_i $ is also in $I$.

Now write $u_i=v_i+w_i$, where $v_i\in E^1_X$ and $w_i=u_i-v_i$, so that $\partial_X(u_i)=\partial v_i$.  Then
\[
g_1=\sum_{i=1}^\ell y_iu_i=\sum_{i=1}^\ell y_i(v_i+w_i)=\sum_{i=1}^\ell y_iv_i+\sum_{i=1}^\ell y_iw_i=h+k.
\]
It is known  that the degree two part of the Orlik-Solomon ideal $I$ decomposes as $I^2=\bigoplus_{Y \in L_2(\A)} I^2_Y$, see \cite[Prop.~3.24]{OT}. We assert that $h=\sum_{i=1}^\ell y_iv_i$ is in $I^2_X$.

For a contradiction, assume otherwise. Then $h$ represents a nontrivial element in $A^2(\A_X)$, the Orlik-Solomon algebra of the subarrangement $\A_X\subset \A$. Then, by the Brieskorn lemma \cite[Cor.~3.27]{OT}, $h$ represents a nontrivial element in $A^2=A^2(\A)= \bigoplus_{Y \in L_2(\A)} A^2(\A_Y)$. Since $k=\sum_{i=1}^\ell y_iw_i \in E^1_X \wedge (E^1 \smallsetminus E^1_X)$, $g_1=h+k$ also represents a nontrivial element in $A^2$, contradicting the fact that $g_1 \in I$.

Using $h \in I^2_X \implies \partial h=0$ and the fact that $\partial y_i=0$ noted above, we have
\[
\partial h = \partial y_1 v_1+ \dots + \partial y_\ell v_\ell - y_1 \partial v_1-\dots - y_\ell \partial v_\ell = -y_1 \partial v_1-\dots - y_\ell \partial v_\ell =0.
\]
The linear independence of the set $\{y_1,\dots,y_\ell\}$  then yields $\partial v_1=\dots=\partial v_\ell=0$, that is, $\partial_X u_i=0$ for each $i$.  But the polychrome flat $X$ was arbitrary, so $u_1,\dots,u_\ell$ 
satisfy $\partial_X u_i = 0$ for each $i$ and every flat $X$ of $\A$ which is polychrome with respect to $\Pi$. Thus, $u_i \in L$ for each $i$, and $g=\sum_{i=1}^\ell (\nu_i-\nu_{\ell+1})u_i$ is in the ideal $I_L$.
\end{proof}

\begin{rem}
One can alternatively reduce to the case of resonance components supported on the entire arrangement $\A$ (as in the first paragraph of the above proof) using the ideals $I_L$ and $J_L$. Let $\A' \subset \A$ be a proper subarrangement, and write $I^2(\A)=I^2(\A')\oplus I^2(\A,\A')$, where 
\[
\begin{aligned}
I^2(\A')&=\langle \partial e_i e_j e_k \mid H_i\cap H_j \cap H_k \in L_2(\A') \rangle,\\
I^2(\A,\A')&=\langle \partial e_i e_j e_k \mid H_i\cap H_j \cap H_k \in L_2(\A)\ \text{and}\ 
|\{H_i,H_j,H_k\}\cap (\A\smallsetminus\A')| \ge 1 \rangle.
\end{aligned}
\]
If $L$ is a resonance component supported on $\A'$ and $g \in J_L$ is a generator, write $g=g_1+g_2$, where $g_1 \in I^2(\A')$ and 
$g_2 \in I^2(\A,\A')$. One can then use the condition $x g \in I_L$ for any $x \in L$ to show that $g_2=0$.

This chore is facilitated by results on the structure of multinets, namely there are no multinets with more than four blocks, see \cite{BZ} and the references therein.
\end{rem}

In light of Theorem \ref{thm:Ared}, 
Theorem \ref{thm:main} provides a formula for the ranks of the Chen groups of $G$ in terms of the resonance variety $\RR^1(\A)$.  
This formula, $\theta_k(G)=\sum_{m\ge 2} h_m \theta_k(F_m)$ for $k\gg 0$, where $h_m$ is the number of irreducible components of $\RR^1(\A)$ of dimension $m$, for the Chen ranks was conjectured by Suciu \cite{Su}.
In \cite{SS2}, Schenck-Suciu proved that $\theta_k(G)\ge\sum_{m\ge 2} h_m \theta_k(F_m)$ for $k\gg 0$.  
Suciu's original conjecture predicted equality in 
this Chen ranks formula for all $k\ge 4$, but in \cite{SS2} it is shown 
the value for which $\theta_k(G)$ is given by a 
fixed polynomial in $k$ depends on the Castelnuovo-Mumford regularity 
of the linearized Alexander invariant of $G$.

\begin{example} \label{ex:hess}
Let $\A$ be the Hessian arrangement in $\C^3$, defined by the polynomial $Q=xyz\prod_{1\le i,j\le 3}(x+\omega^iz+\omega^jz)$, where $\omega=\exp(2\pi i/3)$. The projectivization of $\A$ consists of the twelve lines in $\C\P^2$ passing
through the nine inflection points of a smooth plane cubic curve. 
Four lines meet at each of the nine inflection points, yielding nine rank $2$ flats in $L(\A)$ of cardinality $4$, and associated $3$-dimensional components of $\RR^1(\A)$.  The arrangement $\A$ has $54$ subarrangements lattice-isomorphic to the rank $3$ braid arrangement. Each of these contributes a $2$-dimensional component to $\RR^1(\A)$.     The arrangement $\A$ itself admits a nontrivial neighborly partition, and has a corresponding $3$-dimensional component of $\RR^1(\A)$. A calculation reveals that these ($10$ $3$-dimensional and $54$ $2$-dimensional) components constitute all irreducible components of $\RR^1(\A)$.  Consequently, if $G$ is the fundamental group of the complement of $\A$, by Theorem \ref{thm:main} we have $\theta_k(G) = 10(k^2-1)+54(k-1)\ \text{for}\ k\gg 0$.
\end{example}

\section{Coxeter arrangements} \label{sec:cox}
In this section, we use Theorem \ref{thm:main} to determine the Chen ranks of the pure braid groups associated to the Coxeter groups of types A, B, and D, proving Theorem \ref{thm:coxeter}.

\begin{example} \label{ex:braid}
Let $\A_n$ be the braid arrangement, the type A Coxeter arrangement in $\C^n$ with hyperplanes $\ker(x_i-x_j)$, $1\le i<j\le n$.  The complement of $\A_n$ is the configuration space of $n$ ordered points in $\C$, with fundamental group  the Artin pure braid group $G=P_n$.  The resonance variety $\RR^1(\A_n)$ has $\binom{n+1}{4}$ two-dimensional irreducible components, see \cite{CScv,Per10}. 
Theorem \ref{thm:main} yields $\theta_k(P_n)=(k-1)\binom{n+1}{4}\ \text{for}\ k\gg 0$, as first calculated in \cite{CS1}.

More generally, let $\Gamma$ be a simple graph on vertex set $\{1,\dots,n\}$, and let $\A_\Gamma$ be the corresponding graphic arrangement in $\C^n$, consisting of the hyperplanes $\ker(x_i-x_j)$ for which $\{i,j\}$ is an edge of $\Gamma$.  The resonance variety $\RR^1(\A_\Gamma)$ has $\kappa_3+\kappa_4$  two-dimensional irreducible components, where $\kappa_m$ denotes the number of complete subgraphs on $m$ vertices in $\Gamma$, see \cite{SS2}.   If $G$ is the fundamental group of the complement of $\A_\Gamma$, Theorem \ref{thm:main} yields $\theta_k(G)=(k-1)(\kappa_3+\kappa_4)\ \text{for}\ k\gg 0$, as first calculated in \cite{SS2}.
\end{example}

\subsection{Resonance of the Coxeter arrangement of type D}
Let $\DD_n$ be the type D Coxeter arrangement in $\C^n$, with $n(n-1)$ hyperplanes $H_{i,j}^{\pm}=\ker(x_i\pm x_j)$, $1\le i<j\le n$.  
The rank $2$ flats of $\DD_n$ are
\begin{equation*} \label{eq:D2flats}
\begin{array}{lll}
H_{i,j}^-\cap H_{i,k}^-\cap H_{j,k}^-,&
H_{i,j}^-\cap H_{i,k}^+\cap H_{j,k}^+,&
H_{p,q}^-\cap H_{p,q}^+,\\[5pt]
H_{i,j}^+\cap H_{i,k}^-\cap H_{j,k}^+,&
H_{i,j}^+\cap H_{i,k}^+\cap H_{j,k}^-,&
H_{p,q}^\pm \cap H_{r,s}^\pm,
\end{array}
\end{equation*}
where $1\le i < j < k \le n$, $1\le p < q \le n$, $1\le r<s\le n$, and $\{p,	q\}\cap\{r,s\}=\emptyset$.
Note that $\DD_n$ has $4\binom{n}{3}$ rank two flats of multiplicity $3$, and $\binom{n}{2}+12\binom{n}{4}$ rank two flats of multiplicity $2$.  
We determine the structure of the 
variety $\RR^1(\DD_n) \subset \k^{n(n-1)}$.
  
Recall that $\A_n$ is the type A Coxeter arrangement in $\C^n$.  Denote the hyperplanes of $\A_n$ by $H_{i,j}^-=\ker(x_i-x_j)$, $1\le i<j\le n$.  
Note that $\A_4$ is lattice-isomorphic to  $\DD_3$.  It is well known that $\A_3$ and $\A_4$ subarrangements give rise to two-dimensional components of the first resonance variety, see Falk \cite{Fa}.  It is also known that $\DD_4$ subarrangements also yield two-dimensional components of the first resonance variety, see Pereira-Yuzvinsky \cite{PY}. These facts yield a number of components of $\RR^1(\DD_n)$, which we record explicitly.  

Each triple $1\le i<j<k\le n$ yields four $\A_3$ subarrangements of $\DD_n$:
\[
1\ \{H_{i,j}^-,H_{i,k}^-,H_{j,k}^-\},\ 2\ \{H_{i,j}^+,H_{i,k}^-,H_{j,k}^+\},\ 
3\ \{H_{i,j}^-,H_{i,k}^+,H_{j,k}^+\},\ 4\ \{H_{i,j}^+,H_{i,k}^+,H_{j,k}^-\}.
\]
Denote the corresponding components of $\RR^1(\DD_n)$ by $U_{i,j,k}^q$, $1\le q \le 4$.
Each such triple also yields the $\A_4$ subarrangement $\{H_{i,j}^{\pm},H_{i,k}^{\pm},H_{j,k}^{\pm}\}$, with corresponding resonance component $V_{i,j,k}$.  

Each 4-tuple $1\le i<j<k<l\le n$ yields eight $\A_4$ subarrangements of $\DD_n$:
\[
\begin{array}{cc}
1&\{H_{i,j}^-,H_{i,k}^-,H_{i,l}^-,H_{j,k}^-,H_{j,l}^-,H_{k,l}^-\}\\
3&\{H_{i,j}^-,H_{i,k}^+,H_{i,l}^+,H_{j,k}^+,H_{j,l}^+,H_{k,l}^-\}\\
5&\{H_{i,j}^+,H_{i,k}^-,H_{i,l}^+,H_{j,k}^+,H_{j,l}^-,H_{k,l}^+\}\\
7&\{H_{i,j}^+,H_{i,k}^+,H_{i,l}^-,H_{j,k}^-,H_{j,l}^+,H_{k,l}^+\}
\end{array}
\qquad 
\begin{array}{cc}
2&\{H_{i,j}^+,H_{i,k}^+,H_{i,l}^+,H_{j,k}^-,H_{j,l}^-,H_{k,l}^-\}\\
4&\{H_{i,j}^-,H_{i,k}^-,H_{i,l}^+,H_{j,k}^+,H_{j,l}^+,H_{k,l}^-\}\\
6&\{H_{i,j}^-,H_{i,k}^+,H_{i,l}^-,H_{j,k}^+,H_{j,l}^-,H_{k,l}^+\}\\
8&\{H_{i,j}^+,H_{i,k}^-,H_{i,l}^-,H_{j,k}^-,H_{j,l}^+,H_{k,l}^+\}
\end{array}
\]
Denote the corresponding components of $\RR^1(\DD_n)$ by $V_{i,j,k,l}^q$, $1\le q \le 8$.  
Each such 4-tuple also yields the subarrangement $\{H_{i,j}^{\pm},H_{i,k}^{\pm},H_{i,l}^{\pm},H_{j,k}^{\pm},H_{j,l}^{\pm},H_{k,l}^{\pm}\}$ isomorphic to $\DD_4$, with corresponding resonance component $W_{i,j,k,l}$. 

\begin{thm} \label{thm:Dres}
The first resonance variety of the arrangement $\DD_n$ is given by
\[
\RR^1(\DD_n)= \bigcup_{i<j<k}\Bigl(V_{i,j,k}\cup \bigcup_{q=1}^4 U_{i,j,k}^q\Bigr)\cup 
\bigcup_{i<j<k<l}\Bigl(W_{i,j,k,l}\cup \bigcup_{q=1}^8 V_{i,j,k,l}^q\Bigr).
\]
Hence, $\RR^1(\DD_n)\subset \k^{n(n-1)}$ is a union of $5\binom{n}{3}+9\binom{n}{4}$ two-dimensional components. 
\end{thm}
\begin{proof}
The inclusion of the union in $\RR^1(\DD_n)$ follows from the preceding discussion, so it suffices to establish the opposite inclusion.  For this, it is enough to show that a subarrangement of $\DD_n$ not isomorphic to $\A_3$, $\A_4$, or $\DD_4$ does not contribute a component to $\RR^1(\DD_n)$.  Let $\BB$ be such a subarrangement.  We show that $\BB$ does not admit a nontrivial neighborly partition.

If $\BB\subset \DD_n$ is a subarrangement of cardinality at most $3$ that is not isomorphic to $\A_3$, then $\BB$ is in general position, and admits no nontrivial neighborly partition.  So we may assume that $|\BB|\ge 4$.

Let $\Pi$ be a neighborly partition of $\BB$.  If $\BB$ contains $3$ hyperplanes $H_{i,j}^a,H_{k,l}^b,H_{r,s}^c$, where $a,b,c\in\{+,-\}$, with $|\{i,j,k,l,r,s\}|\ge 5$, we assert that $\Pi$ must be trivial.  
If $|\{i,j,k,l,r,s\}|=6$, then the hyperplanes $H_{i,j}^a,H_{k,l}^b,H_{r,s}^c$ are in general position, so must lie in the same block, say $\Pi_0$, of $\Pi$.  Let $H_{p,q}^d$ be any other hyperplane in $\BB$, where $d\in\{+,-\}$. Then $H_{p,q}^d$ must be in general position with at least one of 
$H_{i,j}^a,H_{k,l}^b,H_{r,s}^c$, which implies that $H_{p,q}^d \in\Pi_0$.  Thus, $\Pi=\Pi_0$ is trivial.

If $H$ and $H'$ are hyperplanes of an arrangement $\BB$ such that the rank two flat $H \cap H'$ is not contained in any other hyperplane $H''$ of $\BB$, i.e., $\codim H\cap H'\cap H''>2$, we will write $H \pitchfork H'$ in $\BB$.

If $|\{i,j,k,l,r,s\}|= 5$, permuting indices if necessary, we can assume that $\{i,j\}=\{1,2\}$, $\{k,l\}=\{3,4\}$, and $\{r,s\}=\{4,5\}$.  Let $\Pi_0$ be the block of $\Pi$ containing $H_{1,2}^a$.  Then, $H_{1,2}^a \pitchfork H_{3,4}^b$ and $H_{1,2}^a\pitchfork H_{4,5}^c$ in $\BB$, which implies that $H_{3,4}^b,H_{4,5}^c\in \Pi_0$.

Let $H_{p,q}^d$ be any other hyperplane in $\BB$.  We must show that $H_{p,q}^d\in\Pi_0$.  If $\{p,q\}=\{1,2\}$ or $3\le p$, then $H_{1,2}^a \pitchfork H_{p,q}^d$ is a rank $2$ flat of $\BB$, which implies that $H_{p,q}^d\in\Pi_0$.  
So we may assume that $p\in\{1,2\}$ and $q\ge 3$.  If $q\ge 5$, then $H_{3,4}^b\pitchfork H_{p,q}^d$ in $\BB\implies H_{p,q}^d\in \Pi_0$.  Similarly, if $q=3$, then $H_{4,5}^c\pitchfork H_{p,q}^d$ in $\BB \implies H_{p,q}^d\in \Pi_0$. 
It remains to consider the instance $p\in\{1,2\}$ and $q=4$.

There is a multiplicity $3$, rank $2$ flat $H_{p,3}^e \cap H_{p,4}^d \cap H_{3,4}^b \in L_2(\DD_n)$, for some $e\in\{+,-\}$.  Note that $H_{p,3}^e \pitchfork H_{4,5}^b$ in $\DD_n$.  If $H_{p,3}^e \notin\BB$, then $H_{p,4}^d \pitchfork H_{3,4}^b$ in $\BB \implies H_{p,4}^d\in \Pi_0$.  If $H_{p,3}^e \in\BB$, then $H_{p,3}^e \pitchfork H_{4,5}^b$ in $\BB\implies H_{p,3}^d\in \Pi_0$.  In this last instance, the flat $H_{p,3}^e \cap H_{p,4}^d \cap H_{3,4}^b \in L_2(\BB)$ must be monochrome since $\Pi$ is neighborly, which forces $H_{p,4}^d\in \Pi_0$.

We are left with the case where any triple of hyperplanes $H_{i,j}^a,H_{k,l}^b,H_{r,s}^c$ in $\BB$ satisfies $|\{i,j,k,l,r,s\}|\le 4$, which implies that 
all hyperplanes of $\BB \subset\DD_n$ involve at most $4$ indices.  Thus, $\BB$ is a proper subarrangement of $\DD_4$ (with $|\BB|\ge 4$).  If $\BB \subsetneq \DD_3 \cong \A_4$ (for any of the $9$ choices of $\A_4$ subarrangements of $\DD_4$), it is readily checked that $\BB$ admits no
 nontrivial neighborly partition. So we may assume that $\A_4 \subsetneq \BB \subsetneq \DD_4$, for some choice of $\A_4 \subset \DD_4$.
 
For such $\BB$, there are pairs of indices $i<j$ and $k<l$ so that $|\{H_{i,j}^\pm\}\cap \BB|=1$ and $|\{H_{k,l}^\pm\}\cap \BB|=2$. If $k<l$ is the only pair of indices with $|\{H_{k,l}^\pm\}\cap \BB|=2$, then $\BB=\A_4 \cup \{H_{k,l}^\pm\}\subsetneq \DD_4$ is an arrangement of $7$ hyperplanes containing a copy of $\A_4$.  Checking that the hyperplane $H_{k,l}^b \in \BB\smallsetminus\A_4$ is transverse to $\A_4$, we see that $\BB$ admits no nontrivial neighborly partition.

Consequently, we may assume that there is more than one pair of indices $k<l$ with $|\{H_{k,l}^\pm\}\cap \BB|=2$.  At least one such a pair satisfies $|\{i,j\}\cap \{k,l\}|=1$.  Write $H_{i,j}^a \in \BB$ and $H_{i,j}^{\bar{a}}\notin\BB$, and assume that $\{H_{j,k}^a,H_{j,k}^{\bar{a}}\}\subset\BB$, where $\{a,\bar{a}\}=\{+,-\}$. (The other cases are similar, and left to the reader.)  As before, let $\Pi_0$ be the block of $\Pi$ containing $H_{i,j}^a$.  

Suppose $a=+$.  Since $\A_4 \subsetneq\BB$, at least one of $H_{i,k}^+,H_{i,k}^-$ is in $\BB$. In $L_2(\DD_n)$ there are rank $2$ flats
\[
H_{i,j}^+\cap H_{i,k}^+ \cap H_{j,k}^-,\ H_{i,j}^+\cap H_{i,k}^- \cap H_{j,k}^+,\ 
H_{i,j}^-\cap H_{i,k}^- \cap H_{j,k}^-,\ H_{i,j}^-\cap H_{i,k}^+ \cap H_{j,k}^+.
\]
Since $H_{i,j}^-\notin\BB$, the last two of these yield flats of multiplicity $2$ in $L_2(\BB)$ (or nothing in $L_2(\BB)$).  Consequently, if $H_{i,k}^b \in \BB$, then $H_{i,k}^b,H_{j,k}^+,H_{j,k}^-$ all lie in the same block of $\Pi$.  Then $H_{i,j}^+\cap H_{i,k}^b\cap H_{j,k}^{\bar{b}}\in L_2(\BB)$, where $\{b,\bar{b}\}=\{+,-\}$, and this flat must be monochrome.  So $H_{i,j}^+,H_{i,k}^b,H_{j,k}^+,H_{j,k}^-\in\Pi_0$.  From this, it follows easily that $\Pi_0=\Pi$, and hence $\Pi$ is trivial.

A similar argument shows that $\Pi$ is trivial if $a=-$, completing the proof.
\end{proof}
\subsection{Resonance of the Coxeter arrangement of type B}
Let $\BB_n$ be the type B Coxeter arrangement in $\C^n$, consisting of the $n^2$ hyperplanes $H_i^{}=\ker(x_i)$, $1\le i\le n$, and $H_{i,j}^{\pm}=\ker(x_i\pm x_j)$, $1\le i<j\le n$.  
The rank $2$ flats of $\BB_n$ are
\begin{equation*} \label{eq:B2flats}
\begin{array}{llll}
H_{i,j}^-\cap H_{i,k}^-\cap H_{j,k}^-,&
H_{i,j}^-\cap H_{i,k}^+\cap H_{j,k}^+,& H^{}_i \cap H_{j,k}^{\pm},&
H_p^{} \cap H_q^{} \cap H_{p,q}^-\cap H_{p,q}^+,\\[5pt]
H_{i,j}^+\cap H_{i,k}^-\cap H_{j,k}^+,&
H_{i,j}^+\cap H_{i,k}^+\cap H_{j,k}^-,& H^{}_j \cap H_{i,k}^{\pm},&
 H^{}_k \cap H_{i,j}^{\pm},\quad H_{p,q}^\pm \cap H_{r,s}^\pm,
\end{array}
\end{equation*}
where $1\le i < j < k \le n$, $1\le p < q \le n$, $1\le r<s\le n$, and $\{p,q\}\cap\{r,s\}=\emptyset$.
Note that $\BB_n$ has $\binom{n}{2}$ rank two flats of multiplicity $4$, $4\binom{n}{3}$ rank two flats of multiplicity $3$, and $6\binom{n}{3}+12\binom{n}{4}$ rank two flats of multiplicity $2$.
We determine the structure of the 
variety $\RR^1(\BB_n) \subset \k^{n^2}$.

Since $\DD_n \subset \BB_n$, there is an inclusion $\RR^1(\DD_n) \subset \RR^1(\BB_n)$.  
As noted previously, $\A_3$ and $\A_4$ subarrangements give rise to two-dimensional components of the first resonance variety.  It is also known that $\BB_2$ and $\BB_3$ subarrangements yield resonance components, of dimensions $3$ and $2$ respectively, see \cite{Fa}.  These facts yield a number of components of $\RR^1(\BB_n)$, which we now specify.  

Each 2-tuple $1\le i<j\le n$ yields a subarrangement $\BB_2(i,j)$ of $\BB_n$, isomorphic to $\BB_2$, and a corresponding rank two flat $H_i^{} \cap H_j^{} \cap H_{i,j}^-\cap H_{i,j}^+$.  
Let $L_{i,j}$ be the three-dimensional component of $\RR^1(\BB_n)$ corresponding to this flat (resp., to $\BB_2(i,j)$).

Each 3-tuple $1\le i<j<k\le n$ determines a subarrangement $\BB_3(i,j,k)$ of $\BB_n$, defined by $x_i^{}x_j^{}x_k^{}(x_i^2-x_j^2)(x_i^2-x_k^2)(x_j^2-x_k^2)$, which is isomorphic to $\BB_3$.  This subarrangement yields $12$ two-dimensional components of $\RR^1(\BB_n)$, $11$ corresponding to $\A_4$ subarrangements of $\BB_3(i,j,k)$, and $1$ component corresponding to $\BB_3(i,j,k)$ itself.  The $\A_4$ subarrangements of $\BB_3(i,j,k)$ are
\[
\begin{array}{rlrl}
1&\{H_i^{},H_j^{},H_k^{},H_{i,j}^-,H_{i,k}^-,H_{j,k}^-\},&2&\{H_i^{},H_j^{},H_k^{},H_{i,j}^-,H_{i,k}^+,H_{j,k}^+\},\\
3&\{H_i^{},H_j^{},H_k^{},H_{i,j}^+,H_{i,k}^-,H_{j,k}^+\},&4&\{H_i^{},H_j^{},H_k^{},H_{i,j}^+,H_{i,k}^+,H_{j,k}^-\},\\
5&\{H_i^{},H_{i,j}^-,H_{i,j}^+,H_{i,k}^-,H_{i,k}^+,H_{j,k}^-\},&6&\{H_i^{},H_{i,j}^-,H_{i,j}^+,H_{i,k}^-,H_{i,k}^+,H_{j,k}^+\},\\
7&\{H_j^{},H_{i,j}^-,H_{i,j}^+,H_{i,k}^-,H_{j,k}^-,H_{j,k}^+\},&8&\{H_j^{},H_{i,j}^-,H_{i,j}^+,H_{i,k}^+,H_{j,k}^-,H_{j,k}^+\},\\
9&\{H_k^{},H_{i,j}^-,H_{i,k}^-,H_{i,k}^+,H_{j,k}^-,H_{j,k}^+\},\quad&10&\{H_k^{},H_{i,j}^+,H_{i,k}^-,H_{i,k}^+,H_{j,k}^-,H_{j,k}^+\},\\
11&\{H_{i,j}^-,H_{i,j}^+,H_{i,k}^-,H_{i,k}^+,H_{j,k}^-,H_{j,k}^+\}.
\end{array}
\]
Note that only the last of these is contained in $\DD_n$.  For the $10$ other $\A_4$ subarrangements of $\BB_3(i,j,k)$, let $Y_{i,j,k}^{q}$, $1\le q \le 10$, be the 
corresponding components of $\RR^1(\BB_n)$.  
Let $Z_{i,j,k}$ be the component of  $\RR^1(\BB_n)$ corresponding to $\BB_3(i,j,k)$ itself.  

\begin{thm} \label{thm:Bres}
The first resonance variety of the arrangement $\BB_n$ is given by
\[
\RR^1(\BB_n) = \RR^1(\DD_n) \cup \bigcup_{i<j} L_{i,j} \cup \bigcup_{i<j<k}\Bigl(Z_{i,j,k} \cup \bigcup_{q=1}^{10} Y_{i,j,k}^q\Bigr).
\]
Hence, $\RR^1(\BB_n)\subset \k^{n^2}$ is a union of  $16\binom{n}{3}+9\binom{n}{4}$ two-dimensional components and $\binom{n}{2}$ three-dimensional components. 
\end{thm}
\begin{proof}
The inclusion of the union in $\RR^1(\BB_n)$ follows from the preceding discussion, so it suffices to establish the opposite inclusion.  For this, it is enough to show that a subarrangement of $\BB_n$ not isomorphic to $\A_3$, $\A_4$, $\BB_2$, $\BB_3$, or $\DD_4$ does not contribute a component to $\RR^1(\BB_n)$.  Let $\A$ be such a subarrangement.  We show that $\A$ does not admit a nontrivial neighborly partition.  

In light of Theorem \ref{thm:Dres}, we may assume that $\A$ is not contained in any $\DD_k$ subarrangement of $\BB_n$ for $k \le n$.  Thus, $\A=\A' \cup \A''$, where $\A' \subset \{H_i^{}\}$ is nonempty, and $\A''\subset\{H_{i,j}^{\pm}\}$.  If $|\A''|\le 2$, it is readily checked that $\A\not\cong\A_3,\BB_2$ admits no nontrivial neighborly partition.

Suppose $\A'' \supset\{H_{i,j}^a,H_{k,l}^b,H_{r,s}^c\}$.  If $|\{i,j,k,l,r,s\}|\ge 5$, these three hyperplanes must lie in the same block $\Pi_0$ of a neighborly partition $\Pi$ of $\A$.  For each hyperplane $H_p^{}\in\A'\subset\A$, one of these three and $H_p$ forms a rank two flat of $\A$ of multiplicity two, which implies that $H_p^{}\in\Pi_0$ as well.  Then, arguing as in the proof of Theorem \ref{thm:Dres} reveals that $\Pi$ is trivial.

We are left with the case where $\A=\A'\cup \A''$, $|\A'|\ge 1$, $|\A''|\ge 3$, and all hyperplanes of $\A'' \subset\DD_n$ involve at most $4$ indices, say $\{i,j,k,l\}$.  Assume first that $H^{}_m\in\A'$ for some $m\notin\{i,j,k,l\}$. Then  $L_2(\A)$ contains the flats $H^{}_m \cap H_{r,s}^a$ for each $\{r,s\}\subset\{i,j,k,l\}$ for which $H_{r,s}^a \in \A''$. If $\Pi$ is a neighborly partition of $\A$, it follows that $H_m^{}$ and $H_{r,s}^a$ lie in the same  block $\Pi_0$ of $\Pi$ for all such $\{r,s\}$.  If $H_t^{}\in\A'$ for $t \in\{i,j,k,l\}$, then since $m \notin\{i,j,k,l\}$, $H_m^{} \cap H_t^{}$ is a multiplicity two flat of $\A$, which implies that $H_t^{}\in\Pi_0$ as well and $\Pi$ is trivial.

Consequently, we can assume that all hyperplanes of $\A=\A'\cup\A''$ involve only the indices $\{i,j,k,l\}$, so $\A$ is a subarrangement of the $\BB_4$ arrangement involving these indices. 
We consider the various possibilities for $|\A'|$.

If $|\A'|=4$, then $H^{}_i,H^{}_j,H^{}_k,H^{}_l\in\A$.  Suppose, without loss, that $H_{k,l}^a\in \A$.  Then $H^{}_i\pitchfork H_{k,l}^a$ and $H^{}_j\pitchfork H_{k,l}^a$ are rank two flats of $\A$, so $H^{}_i,H^{}_j,H^a_{k,l}$ must lie in the same block $\Pi_0$ of a neighborly partition $\Pi$ of $\A$.  If $H^b_{i,q}$ or $H^c_{j,q}$ are in $\A$, for $q\in\{k,l\}$, then $H^{}_j\pitchfork H_{i,q}^b, H^{}_i\pitchfork H_{i,q}^c \in L_2(\A)$, which implies $H^b_{i,q}, H^c_{j,q} \in\Pi_0$.  The rank two flat $H^{}_i \cap H^{}_q \cap H^-_{i,q} \cap H^+_{i,q} \in L_2(\BB_n)$  yields a rank two flat in $\A$.  Since $H_i^{},H^-_{i,q},H^+_{i,q} \in \Pi_0$ (if either of the latter two hyperplanes are in $\A$), this flat in $\A$ must be monochrome.  Hence, $H_k^{},H_l^{} \in \Pi_0$, and $\Pi$ is trivial.

If $|\A'|=3$, we can assume that $H^{}_i,H^{}_j,H^{}_k, H_{k,l}^a\in \A$ and $H_l^{}\notin\A$.  As above, $H^{}_i,H^{}_j,H^a_{k,l}\in\Pi_0$ must lie in the same block of a neighborly partition $\Pi$ of $\A$.  Using rank two flats of $\A$ of multiplicity two, as in the previous case, any hyperplane $H^b_{r,s} \in \A$ must also lie in $\Pi_0$.  Since $H_l^{} \notin \A$, the flat $H^{}_k \cap H^{}_l \cap H^-_{k,l} \cap H^+_{k,l} \in L_2(\BB_n)$ yields a flat of multiplicity two or three in $\A$, which must be monochrome.  
Hence, $H_k^{}\in \Pi_0$, and $\Pi$ is trivial.

If $|\A'|=2$, assume that $H_i^{},H_j^{}\in\A$ and $H_k^{},H_l^{}\notin\A$.  If $H^a_{k,l} \in \A$, then 
$H^{}_i,H^{}_j,H^a_{k,l}\in\Pi_0$ must lie in the same block of a neighborly partition $\Pi$ of $\A$, and one can show that $\Pi$ must be trivial by considering multiplicity two rank two flats of $\A$ as above.  So assume that $\{H^\pm_{k,l}\}\cap \A=\emptyset$.  Since the hyperplanes of $\A$ involve all four indices $i,j,k,l$, there are hyperplanes $H_{p,k}^a,H_{q,l}^b \in \A$ with $\{p,q\}=\{i,j\}$.  There is a corresponding multiplicity two rank two flat $H_{p,k}^a\pitchfork H_{q,l}^b \in L_2(\A)$ (as $H_{k,l}^c\notin\A)$.  So $H_{p,k}^a,H_{q,l}^b\in\Pi_0$  must lie in the same block of a neighborly partition $\Pi$ of $\A$.  Since $H^{}_q \pitchfork H_{p,k}^a,H^{}_p\pitchfork H_{q,l}^b \in L_2(\A)$, we have $H_i^{},H_j^{}\in\Pi_0$, and  it follows that $\Pi$ must be trivial.

If $|\A'|=1$, assume that $H_i^{}\in\A$ and $H_j^{},H_k^{},H_l^{}\notin\A$.  If all hyperplanes of $\A$ involve index $i$, then there are $a,b,c\in\{+,-\}$ so that $H_{i,j}^a,H_{i,k}^b,H_{i,l}^c\in\A$.  These hyperplanes are in general position in $\A$, so lie in the same block $\Pi_0$ of a neighborly partition $\Pi$ of $\A$.  Write, for instance, $\{a,\bar{a}\}=\{+,-\}$.  If $H_{i,j}^{\bar{a}}\in\A$, 
then $H_{i,j}^{\bar{a}} \in \Pi_0$ since  $H_{i,j}^{\bar{a}}\pitchfork H_{i,k}^b \in L_2(\A)$. 
Similarly, 
$H_{i,k}^{\bar{b}},H_{i,l}^{\bar{c}} \in \Pi_0$ if either of these hyperplanes is in $\A$. 
Since $H_j\notin\A$, the flat $H_i^{} \cap H_j^{} \cap H_{i,j}^a \cap H_{i,j}^b \in L_2(\BB_n)$ yields a multiplicity two or three flat of $\A$, which must be monochrome. So $H_i^{}\in \Pi_0$, and $\Pi$ is trivial.

Finally, if $|\A'|=1$ and $\A$ has a hyperplane which does not involve the index $i$, 
we can assume that $H^a_{j,k},H^b_{p,l} \in \A$, where $p\in\{i,j,k\}$.  Then $H_i \pitchfork H_{j,k}^a \in L_2(\A)$, so $H_i,H^a_{j,k} \in\Pi_0$ lie in the same block of a neighborly partition $\Pi$ of $\A$. If $p\neq i$, then $H_i \pitchfork H_{p,l}^b \in L_2(\A)$, while if $p=i$, then $H^a_{j,k} \pitchfork H_{p,l}^b \in L_2(\A)$.  So $H^b_{p,l} \in \Pi_0$ as well.  If $H_{r,s}^c \in \A$ and $i \neq r$, then $H_i \pitchfork H_{r,s}^c \in L_2(\A)$, and $H_{r,s}^c \in \Pi_0$.  If $H_{i,l}^c \in \A$, then $H^a_{j,k} \pitchfork H_{i,l}^c \in L_2(\A)$, and $H_{i,l}^c\in\Pi_0$.  Suppose $H_{i,s}^c \in \A$ for $s\in\{j,k\}$.  If $\{i,s\} \cup \{p,l\}=\{i,j,k,l\}$, then  
$H^c_{i,s} \pitchfork H_{p,l}^b \in L_2(\A)$, and $H_{i,s}^c\in\Pi_0$.  Otherwise, we have either $p=i$ or $p=s$.  If $p=i$, there is a flat $H_{i,s}^c \cap H^b_{p,l} \cap H_{s,l}^d \in L_2(\BB_n)$.  If $H_{s,l}^d \in \A$, then $H_{s,l}^d \in \Pi_0$ since $i\neq s$.  Consequently, this flat yields a flat of multiplicity two or three in $\A$, which must be monochrome, so $H_{i,s}^c \in \Pi_0$.  If $p=s$, 
there is a flat $H_{i,s}^c \cap H^b_{p,l} \cap H_{i,l}^d \in L_2(\BB_n)$.  If $H_{i,l}^d \in \A$, then $H_{i,l}^d \in \Pi_0$ since $H^a_{j,k} \pitchfork H_{i,l}^d \in L_2(\A)$.  Consequently, this flat yields a flat of multiplicity two or three in $\A$, which must be monochrome, so $H_{i,s}^c \in \Pi_0$.  Thus, $\Pi$ is trivial.
\end{proof}

\begin{proof}[Proof of Theorem \ref{thm:coxeter}] 
See Example \ref{ex:braid} for the type A pure braid group. 
Let $PB_n=\pi_1(M(\BB_n))$ and $PD_n=\pi_1(M(\DD_n))$ be the type B and D pure braid groups. 
The resolution of the Chen ranks conjecture and the determination of the resonance varieties of the arrangements $\BB_n$ and $\DD_n$ yield the Chen ranks of these groups. The remaining portions of Theorem \ref{thm:coxeter} are immediate consequences of Theorem \ref{thm:main}, Theorem \ref{thm:Ared}, Theorem \ref{thm:Dres}, and Theorem \ref{thm:Bres}.
\end{proof}

\begin{ack}   Calculations with 
Macaulay~2 \cite{GS} were essential to our work. Our collaboration began
at Hiroaki Terao's 60th birthday conference at the Pacific Institute for the Mathematical Sciences. Parts of this work were carried out when the first author visited the University of Illinois, Urbana-Champaign (Spring 2012), and when both authors visited the Mathematisches Forschungsinstitut Oberwolfach (Fall 2012).  We thank PIMS, UIUC, and the MFO for their support and hospitality, and for providing productive mathematical environments.
\end{ack}

\bibliographystyle{amsalpha}

\begin{thebibliography}{00}

\bibitem{Ao} K.~Aomoto, 
{\em Un th\'eor\`eme du type de Matsushima-Murakami 
concernant l'int\'egrale des fonctions multiformes}, 
J. Math. Pures Appl. \textbf{52}  (1973), 1--11. 

\bibitem{BZ} J. Bartz, S. Yuzvinsky, 
\emph{Multinets in $\P^2$}, in: Bridging algebra, geometry, and topology, 21--35, 
Springer Proc. Math. Stat., 96, Springer, Cham, 2014.

\bibitem{bp} B.~Berceanu, S.~Papadima, 
\emph{ Universal representations of braid and braid-permutation groups}, 
 J. Knot Theory Ramifications \textbf{18} (2009), 999--1019.
 
 \bibitem{BL} A. Brownstein, R. Lee, 
\emph{ Cohomology of the group of motions of n strings in 3-space}, in: \emph{ Mapping class groups and moduli spaces of Riemann surfaces} (Goettingen, 1991 / Seattle, WA, 1991), pp. 51--61, Contemp. Math. \textbf{150},  Amer. Math. Soc., Providence, RI, 1993. 

\bibitem{Ch51} K.~T.~Chen,
{\em Integration in free groups},
Ann. of Math. \textbf{54} (1951), 147--162. 

\bibitem{Coh}  D.~Cohen,  
{\em Resonance of basis-conjugating automorphism groups}, Proc. Amer. Math. Soc. 
\textbf{137} (2009), 2835--2841.

\bibitem{DCcompo} D. Cohen, \emph{Cohomology rings of almost-direct products of free groups}, Compositio Math. \textbf{146} (2010), 465--479.

\bibitem{CPVW} F. Cohen, J. Pakianathan, V. Vershinin, J. Wu, \emph{Basis-conjugating automorphisms of a free group and associated Lie algebras}, in: \emph{Groups, homotopy and configuration spaces} (Tokyo 2005), pp. 147--168, Geom. Topol. Monogr. \textbf{13}, Geom. Topol., Coventry, 2008.

\bibitem{CS1} D.~Cohen, A.~Suciu,
{\em The Chen groups of the pure braid group},
The \v Cech centennial (Boston, MA, 1993), 45--64,
Contemp. Math., 181, Amer. Math. Soc., Providence, RI, 1995.

\bibitem{CSai} D.~Cohen, A.~Suciu,
{\em Alexander invariants of complex hyperplane arrangements},
Trans. Amer. Math. Soc.  \textbf{351} (1999), 4043--4067.

\bibitem{CScv} D.~Cohen, A.~Suciu,
{\em Characteristic varieties of arrangements},
Math. Proc. Cambridge Phil. Soc. \textbf{127} (1999), 33--53.

\bibitem{Dahm} D. Dahm, 
\emph{A generalization of a braid theory}, Ph.D. Thesis, Princeton Univ., 1962. 

\bibitem{Del} P. Deligne, \emph{Th\'eorie de Hodge, {\rm II}}, Inst. Hautes \'Etudes Sci. Publ. Math. \textbf{40} (1971), 5--57; \emph{{\rm III}}, \textbf{44} (1974), 5--57.

\bibitem{DPS} A. Dimca, S. Papadima, A. Suciu, 
\emph{Topology and geometry of cohomology jump loci}, 
Duke Math. J.  \textbf{148} (2009), 405--457.

\bibitem{EPY} D.~Eisenbud,  S.~Popescu, S.~Yuzvinsky,
{\em Hyperplane arrangement cohomology and monomials in
the exterior algebra}, Trans. Amer. Math. Soc. \textbf{355}  (2003),  
4365--4383.  

\bibitem{ESV} H.~Esnault, V. ~Schechtman, E. ~Viehweg,
{\em Cohomology of local systems on the complement of hyperplanes}, 
Invent. Math. \textbf{109} (1992), 557-561.  

\bibitem{Fa} M.~Falk,
{\em Arrangements and cohomology},
Ann. Combin. \textbf{1} (1997), 135--157.

\bibitem{Fa01} M. Falk, 
\emph{Combinatorial and algebraic structure in Orlik-Solomon algebras}, 
European J. Combin. \textbf{22} (2001), 687--698. 

\bibitem{FY} M.~Falk, S. Yuzvinsky, 
{\em Multinets, resonance varieties, and pencils of plane curves},
Compositio Math. \textbf{143} (2007), 1069--1088.

\bibitem{Fox} R.~Fox,
{\em Free differential calculus
{\rm I}}, Ann. of Math. \textbf{57} (1953), 547--560;
{\em {\rm II}}, \textbf{59} (1954), 196--210;
{\em {\rm III}}, \textbf{64} (1956), 407--419.

\bibitem{Goldsmith} D. Goldsmith,  
\emph{ The theory of motion groups}, Michigan Math. J. \textbf{ 28} (1981), 3--17.

\bibitem{GS}  D.~Grayson, M.~Stillman,
{\em Macaulay~2: a software system for research in algebraic 
geometry};  available at \texttt{http://www.math.uiuc.edu/Macaulay2}. 

\bibitem{JMM}
C. Jensen, J. McCammond, J. Meier,  
\emph{ The integral cohomology of the group of loops}, Geom. Topol.  \textbf{10} (2006), 759--784.

\bibitem{LY}  A.~Libgober, S.~Yuzvinsky,
{\em Cohomology of the Orlik-Solomon algebras and local systems},
Compositio Math. \textbf{121} (2000), 337--361. 

\bibitem{Ma}  W.S.~Massey,
{\em Completion of link modules}, Duke Math. J.
\textbf{47} (1980), 399--420.  

\bibitem{MS} D.~Matei, A.~Suciu,
{\em Cohomology rings and nilpotent quotients of real and
complex arrangements}, in: Arrangements--Tokyo 1998, 
Adv. Stud. Pure Math., vol.~27,  Math. Soc. Japan,
Tokyo, 2000, pp.~185--215.

\bibitem{McC} J. McCool, 
\emph{On bases-conjugating automorphisms of free groups}, Can. J. Math. \textbf{ 38} (1986), 1525--1529.

\bibitem{Mor} J. Morgan, \emph{The algebraic topology of smooth algebraic varieties}, 
 Inst. Hautes \'Etudes Sci. Publ. Math. \textbf{48} (1978), 137--204.

\bibitem{OS} P.~Orlik, L.~Solomon,
{\em Combinatorics and topology of complements of
hyperplanes}, Invent. Math. \textbf{56} (1980), 167--189. 

\bibitem{OT}  P.~Orlik, H.~Terao,
{\em Arrangements of hyperplanes}, Grundlehren Math. Wiss.,
vol.~300, Springer-Verlag, New~York-Berlin-Heidelberg, 1992. 

\bibitem{PS} S.~Papadima, A.~Suciu,
{\em Chen Lie algebras}, Int. Math. Res. Not. \textbf{2004:21} 
(2004), 1057--1086.   

\bibitem{PS2} S.~Papadima, A.~Suciu,
 \emph{Algebraic invariants for 
right-angled Artin groups}, Math.\,Ann.\,\textbf{334} (2006), 533--555.  

\bibitem{PS3} S.~Papadima, A.~Suciu,
 \emph{Geometric and algebraic aspects of $1$-formality}, 
 Bull. Math. Soc. Sci. Math. Roumanie \textbf{52} (2009), 355--375. 

\bibitem{Per10}
J. Pereira, \emph{Resonance webs of hyperplane arrangements}, in: Arrangements of Hyperplanes--Sapporo 2009, Adv. Stud. Pure Math., vol. 62, Math. Soc. Japan, Tokyo, 2012, pp. 261--291.

\bibitem{PY} J.~Pereira, S.~Yuzvinsky,
\emph{Completely reducible hypersurfaces in a pencil}, Adv. Math. \textbf{219} (2008), 672--688.

\bibitem{STV} V.~Schechtman, H.~Terao, A.~Varchenko,
{\em Local systems over complements of hyperplanes and the
Kac-Kazhdan condition for singular vectors},
J. Pure Appl. Algebra \textbf{100} (1995), 93--102.

\bibitem{SS1} H.~Schenck, A.~Suciu,
{\em Lower central series and free resolutions of hyperplane arrangements},
Trans. Amer. Math. Soc. \textbf{354} (2002), 3409--3433.

\bibitem{SS2} H.~Schenck, A.~Suciu,
{\em Resonance, linear syzygies, Chen groups, and the
Bernstein-Gelfand-Gelfand correspondence},
Trans. Amer. Math. Soc. \textbf{358} (2006), 2269--2289.

\bibitem{Su} A.~Suciu,
{\em Fundamental groups of line arrangements: Enumerative aspects}, 
in:  Advances in algebraic geometry motivated by physics,
Contemporary Math., vol.~276, Amer. Math. Soc, Providence,
RI, 2001, pp. 43--79.  

\bibitem{Su1} A.~Suciu,
\emph{Fundamental groups, Alexander invariants, and cohomology jumping loci}, 
in: Topology of algebraic varieties and singularities, Contemp. Math., vol. 538, 
Amer. Math. Soc., Providence, RI, 2011, pp. 179--223.

\bibitem{Yuz} S.~Yuzvinsky,
{\em Cohomology of Brieskorn-Orlik-Solomon algebras},
Comm. Algebra \textbf{23} (1995), 5339--5354.  

\bibitem{Yuz1} S.~Yuzvinsky,
{\em Orlik-Solomon algebras in algebra and topology},
Russ. Math. Surveys \textbf{56} (2001), 293--364.

\end{thebibliography}

\end{document}